\documentclass[graybox]{svmult}

\usepackage{makeidx}         
\usepackage{graphicx}        

\usepackage{amssymb,amscd,amsmath,latexsym}
\usepackage[all]{xy}
\usepackage{cases}
\usepackage{enumitem}
\usepackage{color}

\spnewtheorem{definition}[theorem]{Definition}{\bfseries}{\rmfamily}
\spnewtheorem{lemma}[theorem]{Lemma}{\bfseries}{\itshape}
\spnewtheorem{proposition}[theorem]{Proposition}{\bfseries}{\itshape}
\spnewtheorem{corollary}[theorem]{Corollary}{\bfseries}{\itshape}
\spnewtheorem{remark}[theorem]{Remark}{\bfseries}{\rmfamily}
\spnewtheorem{example}[theorem]{Example}{\bfseries}{\rmfamily}

\def\into{ \rightarrowtail }
\def\onto{ \twoheadrightarrow }
\def\splito{ \rightleftarrows }
\newdir{ >}{{}*!/-8pt/@{>}}
\def\EE{ \mathbb{E} }
\def\VV{ \mathbb{V} }
\newcommand{\Eq}{ \ensuremath{\mathrm{Eq}} }
\newdir{ >}{{}*!/-8pt/@{>}}

\begin{document}

\title*{On the naturalness of Mal'tsev categories}
\author{D. Bourn, M. Gran and P.-A. Jacqmin}
\institute{Dominique Bourn \at Universit\'e du Littoral, Laboratoire LMPA, BP 699, 62228 Calais Cedex, France. \email{bourn@univ-littoral.fr}
\and Marino Gran \at Universit\'e catholique de Louvain, IRMP, Chemin du Cyclotron 2, 1348 Louvain-la-Neuve, Belgique. \email{marino.gran@uclouvain.be}
\and Pierre-Alain Jacqmin \at University of Ottawa, Department of Mathematics and Statistics, 150 Louis-Pasteur, Ottawa, ON, K1N 6N5, Canada. \email{pjacqmin@uottawa.ca}}
%
%
\maketitle

\begin{quotation}
`Let us also recall that fundamental progress in homological algebra was achieved by replacing module categories by arbitrary abelian categories. (\dots) It is already clear from~\cite{Lambek92} that, in proving some basic lemmas for Mal'cev varieties, one never uses the fact that the category is varietal, but just that the semantical conditions hold. This suggests that one should investigate a purely categorical notion, generalizing that of an abelian category, to develop non-additive `variable' homological arguments.'

\flushright{A.\ Carboni, J.\ Lambek and M.C.\ Pedicchio in~\cite{CLP}, 1990.}
\end{quotation}

\abstract{Mal'tsev categories turned out to be a central concept in categorical algebra. On one hand, the simplicity and the beauty of the notion is revealed by the wide variety of characterizations of a markedly different flavour. Depending on the context, one can define Mal'tsev categories as those for which `any reflexive relation is an equivalence'; 
`any relation is difunctional'; `the composition of equivalence relations on a same object is commutative'; `each fibre of the fibration of points is unital' or `the forgetful functor from internal groupoids to reflexive graphs is saturated on subobjects'. For a variety of universal algebras, these are also equivalent to the existence in its algebraic theory of a Mal'tsev operation, i.e. a ternary operation $p(x,y,z)$ satisfying the axioms $p(x,x,y)=y$ and $p(x,y,y)=x$. On the other hand, Mal'tsev categories have been shown to be the right context in which to develop the theory of centrality of equivalence relations, Baer sums of extensions, and some homological lemmas such as the denormalized $3 \times 3$ Lemma, whose validity in a regular category is equivalent to the weaker `Goursat property', which has also turned out to be of wide interest.}



\section*{Introduction}

The study of Mal'tsev categories originates from a classical theorem of Mal'tsev in 1954~\cite{Maltsev}. For a variety of universal algebras $\mathbb{V}$ (i.e., a category of models of a finitary one-sorted algebraic theory), he proved that the following conditions are equivalent:
\begin{enumerate}[label=\textnormal{(M\arabic*)}]
\item \label{congruence commutativity} \emph{The composition of congruences is commutative, i.e., for any two congruences $R$ and $S$ on a same algebra in~$\mathbb{V}$, the equality $RS=SR$ holds.}
\item \label{ternary term p} \emph{The theory of $\mathbb{V}$ contains a ternary term $p$ satisfying the equations}
$$p(x,x,y)=y \; \; and \;\;  p(x,y,y)=x.$$
\end{enumerate}
Varieties satisfying these conditions are now commonly called `Mal'tsev varieties' \cite{Smith} (or $2$-permutable varieties), and such a term $p$ a `Mal'tsev operation'. This result has been extended by Lambek a few years later in~\cite{Lambek58} where he proved that these conditions are also equivalent to 
\begin{enumerate}[start=3,label=\textnormal{(M\arabic*)}]
\item \label{any relation is difunctional} \emph{Any homomorphic relation $R$ in~$\mathbb{V}$ is difunctional;}
\end{enumerate}
where difunctionality, a property introduced by Riguet~\cite{Riguet}, is the condition $RR^{\circ}R \leqslant R$.
Findlay~\cite{Findlay} and Werner~\cite{Werner} further characterised Mal'tsev varieties as those satisfying the condition
\begin{enumerate}[start=4,label=\textnormal{(M\arabic*)}]
\item \label{any reflexive relation is a congruence} \emph{Any homomorphic reflexive relation in~$\mathbb{V}$ is a congruence.}
\end{enumerate}

In his papers~\cite{Lambek58} and~\cite{Lambek92}, Lambek generalized classical group theory results, such as Goursat's theorem and Zassenhaus lemma, to Mal'tsev varieties. For instance, Goursat's theorem attests that given a homomorphic relation $R \leqslant A \times B$ in a Mal'tsev variety $\mathbb{V}$, the quotients $AR / R^{\circ}R$ and $RB / RR^{\circ}$ are isomorphic, where $AR=\{b \in B \,|\,\exists a \in A, aRb\}$ and $RB=\{a \in A \,|\,\exists b \in B, aRb\}$.

The proofs of these results relied on the so called calculus of relations and could therefore be transposed to a categorical context, by translating syntactic conditions into semantic ones. Actually, in her thesis~\cite{Meisen} supervised by Lambek, Meisen showed that the equivalence between~\ref{congruence commutativity}, \ref{any relation is difunctional} and \ref{any reflexive relation is a congruence} also holds for any exact category in the sense of Barr~\cite{Barr}. In that context, Mal'tsev categories were introduced by Carboni, Lambek and Pedicchio in~\cite{CLP} where they developed some aspects of homological lemmas in a non-abelian categorical context. 

Although axiom~\ref{congruence commutativity} must be stated in the context of categories with a good factorisation system (as regular categories~\cite{Barr} for instance), axioms~\ref{any relation is difunctional} and~\ref{any reflexive relation is a congruence} still make sense in any finitely complete category and turn out to be still equivalent in that context. Mal'tsev categories were thus defined by Carboni, Pedicchio and Pirovano~\cite{CPP} as finitely complete categories satisfying~\ref{any reflexive relation is a congruence} (and thus~\ref{any relation is difunctional}). A good part of the theory of Mal'tsev categories can be developed in the finitely complete setting.

Besides the homological aspects developed in~\cite{CLP}, the Mal'tsev context revealed itself particularly useful to develop the conceptual categorical notion of centrality of equivalence relations. Smith~\cite{Smith} had already opened the way in the varietal context, and then Janelidze made a further step by establishing a link between commutators and internal categories in Mal'tsev varieties~\cite{Janelidze-cat}. Then, partly based also on the results in~\cite{CPP}, Pedicchio was the first to explicitly consider a categorical approach to commutator theory in exact Mal'tsev categories (with coequalizers) \cite{Ped1, Ped2} and in more general categories~\cite{JanPed}. With the introduction of the notion of \emph{connector}~\cite{BG3}, centrality was expressed in its simplest form, and it could be fully investigated in the (regular) Mal'tsev context.
Finally, in the exact context, the construction of the Baer sum of extensions with abelian kernel equivalence relation emerged in a very natural way~\cite{Bourn3, Bourn5}. 

Several new results were later discovered in Mal'tsev categories in connection with the theory of central extensions~\cite{JK, Everaert, DG}, homological lemmas \cite{Bourn3x3, Cuboid} and, recently, with (non-abelian) embedding theorems~\cite{approx, jacqmin1, jacqmin2, jacqmin3}. 

The homological ambition of the initial project of Carboni, Lambek and Pedicchio~\cite{CLP} was not totally achieved. As a matter of fact, this early work was missing the notion of pointed protomodular category~\cite{Bourn6} which was necessary to define and investigate the conceptual categorical notion of short exact sequences. When furthermore the context is regular, it is then possible to establish the classical Short Five Lemma, the Noether isomorphism theorems, the Chinese Remainder's theorem, the Snake Lemma and the long exact homology sequence~\cite{BB}.

In this article we review some of the most striking features of Mal'tsev categories. 
In the first section, we study Mal'tsev operations and their properties. A partial version of them is used to define the notion of a connector between equivalence relations. The second section is devoted to the definition of Mal'tsev categories in the finitely complete context and relevant examples are given. In the third section, we give many characterizations of Mal'tsev categories showing the richness of the notion. One of them says that the fibres of the fibration of points are unital, which implies the uniqueness of a connector between two equivalence relations on a same object in a Mal'tsev category. In Section~\ref{section Stiffly and naturally Mal'tsev categories}, we study stiffly and naturally Mal'tsev categories. The fifth section is devoted to regular Mal'tsev categories. We give an additional characterization of them in this context via the notion of regular pushouts, which makes the proof of the existence of a Mal'tsev operation in the varietal context easier. In Section~\ref{RegularRel}, we further study regular Mal'tsev categories using the calculus of relations. This naturally brings us to have a glance at the weaker Goursat property, equivalent in the regular context to the denormalized $3 \times 3$ Lemma. We conclude this article with Section~\ref{section Baer sums in Mal'tsev categories} where we study Baer sums of extensions with abelian kernel equivalence relation in efficiently regular Mal'tsev categories.

\section*{Acknowledgements}

This work was partly supported by the collaboration project \emph{Fonds d'Appui \`a l'Internationalisation} (2018-2020) to strengthen the research collaborations in mathematics between the Universit\'e catholique de Louvain and the Universit\'e d'Ottawa. The third author gratefully thanks the National Science and Engineering Research Council for their generous support.

\section{Mal'tsev operations}

Let us start with the Mal'tsev theorem on varieties of algebras.

\begin{theorem}\label{CharacterizationM}
Let $\mathbb{V}$ be a variety of universal algebras. The following statements are equivalent:
\begin{itemize}
\item[\ref{congruence commutativity}] For any two congruences $R$ and $S$ on a same algebra in~$\mathbb{V}$, the equality $RS=SR$ holds.
\item[\ref{ternary term p}] The theory of $\mathbb{V}$ contains a ternary term $p$ satisfying the equations
$$p(x,x,y)=y \; \; and \;\;  p(x,y,y)=x.$$
\item[\ref{any relation is difunctional}] Any homomorphic relation in~$\mathbb{V}$ is difunctional, i.e., $RR^{\circ}R \leqslant R$.
\item[\ref{any reflexive relation is a congruence}] Any homomorphic reflexive relation in~$\mathbb{V}$ is a congruence.
\end{itemize}
\end{theorem}

\begin{proof}
Let us start by proving $\ref{ternary term p} \Rightarrow \ref{any relation is difunctional}$. We consider a homomorphic relation $R \leqslant A \times B$ in $\mathbb{V}$ and elements $a_1, a_2$ in $A$ and $b_1, b_2$ in $B$ such that $a_1Rb_1$, $a_2Rb_1$ and $a_2Rb_2$. We deduce that
$$a_1=p(a_1,a_2,a_2) R p(b_1,b_1,b_2)=b_2$$
proving $RR^{\circ}R \leqslant R$.

We now prove $\ref{any reflexive relation is a congruence} \Rightarrow \ref{congruence commutativity}$. Given two congruences $R$ and $S$ on the same algebra $A$, the composition $RS$ is reflexive and therefore a congruence by assumption. Thus, $RS=R \vee S$ is the supremum of $R$ and $S$ in the lattice of congruences on $A$. Indeed, $R = R1_A \leqslant RS$ and $S \leqslant RS$ and given a congruence $T$ on $A$ such that $R \leqslant T$ and $S \leqslant T$, we have $RS \leqslant TT = T$. Similarly, $SR=S\vee R=R\vee S=RS$ proving the desired commutativity.

The implication $\ref{any relation is difunctional} \Rightarrow \ref{any reflexive relation is a congruence}$ immediately follows from the next lemma and the proof of $\ref{congruence commutativity} \Rightarrow \ref{ternary term p}$ is postponed to Section~\ref{section regular Maltsev categories} (Theorem~\ref{MaltseVar}).
\end{proof}

\begin{lemma}\label{reflexive+difunctional}
A homomorphic relation $R \leqslant A \times A$ is an equivalence relation if and only if it is reflexive and difunctional.
\end{lemma}

\begin{proof}
The `only if part' being trivial, let us prove the `if part'. Assume $R$ is reflexive and difunctional. Let us first prove it is symmetric. If $xRy$ for some elements $x$ and $y$ in $A$, we know by reflexivity that $yRy$, $xRy$ and $xRx$. Thus difunctionality implies $yRx$ which shows the symmetry of $R$. For transitivity, let now $x,y,z \in A$ be such that $xRyRz$. Since $xRy$, $yRy$ and $yRz$, we have $xRz$ also by difunctionality, proving that $R$ is transitive.
\end{proof}

Of course, this lemma can be generalized internally to any finitely complete category using the Yoneda embedding.

\begin{remark}
The equivalence between \ref{ternary term p} and \ref{any relation is difunctional} in Theorem~\ref{CharacterizationM} can be displayed in the form of a matrix as $\left(\begin{array}{ccc|c}x&y&y&x\\u&u&v&v\end{array}\right)$. Reading it vertically, this matrix represents difunctionality of a relation. Indeed, a relation $R\leqslant A\times B$ is difunctional when $xRu$, $yRu$ and $yRv$ (the left columns) imply $xRv$ (the right column). On the other hand, reading the matrix horizontally, the identities $p(x,y,y)=x$ and $p(u,u,v)=v$ of \ref{ternary term p} appear. This phenomenon is not at all a coincidence, and the general theory of such matrices has been introduced in~\cite{janelidzematrices} to understand many properties of varieties of universal algebras related to \emph{Mal'tsev conditions} (such as \ref{ternary term p}) from a categorical perspective.
\end{remark}

Let us now have a closer look at Mal'tsev operations.

\begin{definition}
A \emph{Mal'tsev operation} on a set $X$ is a ternary operation $p: X\times X\times X \to X$ such that the identities $p(x,x,y)=y$ and $p(x,y,y)=x$ hold for any $x,y$ in $X$. A \emph{Mal'tsev algebra} is a set $X$ endowed with a Mal'tsev operation. We denote by $\mathsf{Mal}$ the variety of Mal'tsev algebras (including the empty set $\varnothing$).
\end{definition}

\begin{definition}
Let $X$ be a set and $p\colon X \times X \times X \to X$ a Mal'tsev operation. We say that
\begin{itemize}
\item $p$ is \emph{left associative} if it satisfies the axiom: $$p(p(x,y,z),z,w)=p(x,y,w)$$
\item $p$ is \emph{right associative} if it satisfies the axiom: $$p(x,y,p(y,z,w))=p(x,z,w)$$
\item $p$ is \emph{associative} if it satisfies the axiom: $$p(p(x,y,z),u,v)=p(x,y,p(z,u,v))$$
\item $p$ is \emph{commutative} if it satisfies the axiom: $$p(x,y,z)=p(z,y,x)$$
\item $p$ is \emph{autonomous} if it is a morphism of Mal'tsev algebras $X^3\to X$, i.e., if it satisfies the axiom:
\begin{align*}
&p(p(x_1,y_1,z_1),p(x_2,y_2,z_2),p(x_3,y_3,z_3))\\
&=p(p(x_1,x_2,x_3),p(y_1,y_2,y_3),p(z_1,z_2,z_3))
\end{align*}
\end{itemize}
\end{definition}

\begin{lemma}\label{lemma p cancelative}
Let $p$ be a right associative Mal'tsev operation $p\colon X \times X \times X \to X$ on a set $X$. If $x,y,a,b \in X$ are elements such that
$p(x,y,a)=p(x,y,b)$, then $a=b$.
\end{lemma}

\begin{proof}
It follows from\\
\hspace*{0,4cm} $a=p(y,y,a)=p(y,x,p(x,y,a))=p(y,x,p(x,y,b))=p(y,y,b)=b.$
\end{proof}

\begin{proposition}\label{associative=left+right}
Let $p$ be a Mal'tsev operation $p\colon X \times X \times X \to X$ on a set $X$. Then $p$ is associative if and only if it is left associative and right associative.
\end{proposition}

\begin{proof}
If $p$ is associative, then we can compute $$p(p(x,y,z),z,w)=p(x,y,p(z,z,w))=p(x,y,w)$$ which proves that $p$ is left associative. Right associativity is proved similarly.

If now $p$ is both left associative and right associative, we can compute
$$p(p(x,y,z),u,v)=p(p(x,y,z),z,p(z,u,v))=p(x,y,p(z,u,v))$$
which proves that $p$ is associative.
\end{proof}

\begin{proposition}\label{proposition p autonomous}
Let $p$ be a Mal'tsev operation $p\colon X \times X \times X \to X$ on a set~$X$. Then $p$ is autonomous if and only if it is associative and commutative.
\end{proposition}

\begin{proof}
We first assume that $p$ is autonomous. We can then compute
\begin{align*}
&p(p(x,y,z),u,v)=p(p(x,y,z),p(x,x,u),p(x,x,v))\\
&=p(p(x,x,x),p(y,x,x),p(z,u,v)) =p(x,y,p(z,u,v))
\end{align*}
and
\begin{align*}
&p(x,y,z)=p(p(y,y,x),p(y,y,y),p(z,y,y))\\
&=p(p(y,y,z),p(y,y,y),p(x,y,y)) =p(z,y,x)
\end{align*}
proving associativity and commutativity.

Let us now assume that $p$ is associative and commutative. In that case, we have
\begin{align*}
&p(p(y,z,u),x,p(p(x,y,z),u,v)) = p(p(p(y,z,u),x,p(x,y,z)),u,v)\\
&= p(p(p(y,z,u),y,z),u,v) = p(p(p(u,z,y),y,z),u,v)\\
&= p(p(u,z,z),u,v) = p(u,u,v) = v\\
&=p(p(y,z,u),p(y,z,u),v) = p(p(y,z,u),x,p(x,p(y,z,u),v))
\end{align*}
showing via Lemma~\ref{lemma p cancelative} that
$$p(p(x,y,z),u,v)=p(x,p(y,z,u),v).$$
Using this, together with left and right associativity and commutativity, we have
\begin{align*}
&p(p(x_1,y_1,z_1),p(x_2,y_2,z_2),p(x_3,y_3,z_3))\\
&= p(p(x_1,p(x_2,y_2,z_2),z_1),y_1,p(x_3,y_3,z_3))\\
&= p(p(x_1,x_2,p(y_2,z_2,z_1)),y_1,p(x_3,y_3,z_3))\\
&= p(x_1,x_2,p(p(y_2,z_2,z_1),y_1,p(x_3,y_3,z_3)))\\
&= p(x_1,x_2,p(x_3,y_1,p(p(y_2,z_2,z_1),y_3,z_3)))\\
&= p(x_1,x_2,p(x_3,y_1,p(p(y_2,y_3,z_1),z_2,z_3)))\\
&= p(x_1,x_2,p(x_3,y_1,p(y_2,y_3,p(z_1,z_2,z_3))))\\
&= p(p(x_1,x_2,x_3),y_1,p(y_2,y_3,p(z_1,z_2,z_3)))\\
&= p(p(x_1,x_2,x_3),p(y_1,y_2,y_3),p(z_1,z_2,z_3))
\end{align*}
which concludes the proof.
\end{proof}

Now that we have studied some properties of Mal'tsev operations and how they interplay, we can define the notion of a connector of equivalence relations as in~\cite{BG} (slightly more general than the notion of pregroupoid, due to A.\ Kock~\cite{Kock}, see also~\cite{JohnstoneH}):

\begin{definition}\label{connector}
Let $R$ and $S$ be two equivalence relations on a same object $X$ in a finitely complete category, and consider the following pullback:
$$\xymatrix{R \times_X S \ar[r]^-{p_1^S} \ar[d]_-{p_0^R} & S \ar[d]^-{d_0^S} \\ R \ar[r]_-{d_1^R} & X}$$
A \emph{connector} between $R$ and $S$ is a morphism $p\colon R \times_X S \to X$ satisfying the following axioms:
\begin{enumerate}
\item $xSp(xRySz)$ and $p(xRySz)Rz$;
$$\xymatrix@!0@R=3em@C=3em{x \ar@{-}[rr]^{R} \ar@{-}[rd]_-{S} && y \ar@{-}[rd]^-{S} \\ & p(xRySz) \ar@{-}[rr]_-{R} && z}$$
\item $p$ is a partial Mal'tsev operation, i.e., $p(xRxSy)=y$ and $p(xRySy)=x$;
\item $p$ is left and right associative, i.e., $p(p(xRySz)RzSw)=p(xRySw)$ and $p(xRySp(yRzSw))=p(xRzSw)$.
\end{enumerate}
 In this case we say that the connector $p$ makes $R$ and $S$ \emph{centralize} each other.
\end{definition}

By the (partial version of) Proposition~\ref{associative=left+right}, such a connector is associative, i.e., for any element $x,y,z,u,v \in X$ such that $xRySzRuSv$,
we have $$p(p(xRySz)RuSv)=p(xRySp(zRuSv)).$$

\begin{example}
If $\nabla_X$ represents the largest equivalence relation on $X$, a connector between $\nabla_X$ and $\nabla_X$ is simply an associative Mal'tsev operation on~$X$.
\end{example}

Given an arrow $f \colon T \rightarrow X$ we write ($\Eq[f], p_1,p_2)$ for its \emph{kernel pair} which is underlying an equivalence relation defined by the pullback
 $$
 \xymatrix{\Eq[f] \ar[r]^-{p_2} \ar[d]_{p_1} & T\ar[d]^f \\
 T \ar[r]_f & X.
 }
  $$
  
\begin{example}\label{ddif}
Given a relation $(f,g)\colon T\into X\times Y$ and the associated kernel equivalence relations $\Eq[f]$ and $\Eq[g]$ of $f$ and $g$, respectively, this relation $T$ is difunctional if and only if $\Eq[f]$ and $\Eq[g]$ centralize each other.
\end{example}

\begin{example}\label{exGroupoid}\cite{CPP, Janelidze-cat}
Given a reflexive graph
$$\xymatrix{X_1 \ar@<5pt>[r]^-{d_0} \ar@<-5pt>[r]_-{d_1} & X_0 \ar[l]|-{s_0}}$$
in a finitely complete category and $\Eq[d_0]$ and $\Eq[d_1]$ the kernel equivalence relations of $d_0$ and $d_1$ respectively, connectors between $\Eq[d_0]$ and $\Eq[d_1]$ are in $1$-to-$1$ correspondence with groupoid structures on the reflexive graph.
\end{example}

Considering again two equivalence relations $R$ and $S$ on $X$ in a finitely complete category, we define $R \square S$ via the following pullback
$$\xymatrix{R \square S \ar@{ >->}[rr] \ar@{ >->}[d] && S \times S \ar@{ >->}[d]^-{s \times s} \\ R \times R \ar@{ >->}[r]_-{r \times r} & X^4 \ar[r]_-{\text{tw}_{2,3}} & X^4}$$
where $\text{tw}_{2,3}\colon X^4 \to X^4$ is the isomorphism defined by $\text{tw}_{2;3}(x,y,w,z)=(x,w,y,z)$. In set theoretical terms, $R\square S$ is the set of four-tuples $(x,y,w,z)$ such that $xRy$, $wRz$, $xSw$ and $ySz$, often depicted as:
$$\xymatrix{x \ar@{-}[r]^-{R} \ar@{-}[d]_-{S} & y \ar@{-}[d]^-{S} \\ w \ar@{-}[r]_{R} & z}$$ 
We also consider the factorization
\begin{equation} \label{canonical}
\alpha\colon R \square S \to R \times_X S\colon (x,y,w,z) \mapsto (x,y,z).
\end{equation}
If $R \cap S = \Delta_X$ (the discrete relation on $X$), this factorization is a monomorphism. Indeed, if $(x,y,w,z)$ and $(x,y,w',z)$ are in $R\square S$, then $wRzRw'$ and $wSxSw'$, showing that $w (R\cap S) w'$ and thus $w=w'$.

Moreover, given a connector $p\colon R \times_X S \to X$, we can construct a section for $\alpha\colon R \square S \to R \times_X S$ via
$$R \times_X S \to R \square S\colon (x,y,z) \mapsto (x,y,p(x,y,z),z).$$ These observations lead us to the following proposition.

\begin{proposition}\label{oneconnector}
Given two equivalence relations $R$ and $S$ on the same object $X$ in a finitely complete category, if $R \cap S = \Delta_X$, then there is at most one connector between $R$ and $S$.
\end{proposition}


\section{Mal'tsev categories}

\subsection{Definition and examples}

 As mentioned in the introduction, the definition of a Mal'tsev category is of an undisputable simplicity~\cite{CLP,CPP}:

\begin{definition}
	A category $\EE$ is said to be a \emph{Mal'tsev} one, when it is finitely complete and such that any reflexive relation in $\EE$ is an equivalence relation.
\end{definition}

A typical example of such a category is the category $\mathsf{Gp}$ of groups since it satisfies condition~\ref{ternary term p} of Theorem~\ref{CharacterizationM} with the term $p(x,y,z)=xy^{-1}z$.

The class of examples can be quickly extended thanks to the following straightforward lemma:

\begin{lemma}\label{conser}
	Given a left exact conservative functor $U\colon\EE \to \EE'$ between finitely complete categories, the category $\EE$ is Mal'tsev if $\EE'$ is Mal'tsev.
\end{lemma}

So, considering the forgetful functors to the category $\mathsf{Ab}$ of abelian groups, the category $\mathsf{Rg}$ of rings and, given a ring $A$, any category of $A$-modules and $A$-algebras are immediately Mal'tsev.

The variety $\mathsf{Mal}$ of Mal'tsev algebras produces a Mal'tsev category according to the Mal'tsev theorem. More generally, it is the case for any Mal'tsev variety $\VV$, considering the left exact conservative forgetful functor: $U\colon\VV \to \mathsf{Mal}$. On the other hand, given any category $\EE$, the functor category $\mathcal F(\EE,\VV)$ is clearly a Mal'tsev category as well. 

The variety $\mathsf{Heyt}$ of Heyting algebras is a Mal'tsev variety~\cite{Smith}.
From that, the dual $\mathsf{Set}^{op}$ of the category of sets, and more generally the dual $\EE^{op}$ of any elementary topos $\EE$ is a Mal'tsev category~\cite{CKP}. It is also the case for the dual of the category of compact Hausdorff spaces~\cite{CKP}. Another source of examples is given by the following straightforward observation:

\begin{lemma}
	The notion of Mal'tsev category is stable under slicing and coslicing.
\end{lemma} 

This means that, when $\EE$ is a  Mal'tsev category, so are the slice categories $\EE/Y$ and the coslice categories $Y/\EE$, for any object $Y\in \EE$. Accordingly, any fibre $\mathsf{Pt}_Y\EE$ of the fibration of points is a Mal'tsev category (see the definition before Theorem~\ref{theounit} below).

\subsection{Yoneda embedding for internal structures}

Considering internal algebras of a Mal'tsev variety, one gets a third important source of examples. Any object $X$ in a category $\mathbb E$ produces a functor: $$Y(X)=Hom_{\mathbb E}(-,X)\colon\mathbb E^{op}\to \mathsf{Set}$$ This, in turn, produces a fully faithful functor:
$\;\;\mathbb E \to \mathcal F(\mathbb E^{op},\mathsf{Set})\;\;$
which is called the \emph{Yoneda embedding}. It is left exact when, in addition, the category $\mathbb E$ is finitely complete, the left exactness property being a synthetic translation of the universal properties of the finite limits.

Given any algebraic theory $\mathbb T$ defined by any number of operations with finite arity and any number of axioms, we shall denote by $\VV(\mathbb T)$ the associated variety, and by $\mathbb T(\EE)$ the category of the internal $\mathbb T$-algebras in the finitely complete category $\EE$. Then there is a canonical factorization $Y_{\mathbb T}$ making the following diagram commute:
\[\xymatrix@C=3pc@R=3pc{  \mathbb T(\EE) \ar@{-->}[rr]^{Y_{\mathbb T}} \ar[d]_{\mathcal U_{\mathbb T}^{\EE}} && \mathcal F(\mathbb E^{op},\VV(\mathbb T)) \ar[d]^{\mathcal F( \mathbb E^{op}, \mathcal U_{\mathbb T})} \\
	\mathbb E \ar[rr]_Y && \mathcal F(\mathbb E^{op},\mathsf{Set})}\]
where $\mathcal U_{\mathbb T}$ and $\mathcal U_{\mathbb T}^{\EE}$ are the induced forgetful functors, which are both left exact and conservative. Accordingly they are faithful and reflect finite limits as well.

\begin{proposition}\label{Ystruct1}
	The functor $Y_{\mathbb T}$ is fully faithful and left exact.
\end{proposition}

\begin{proof}
The faithfulness is straightforward.  Now, given a pair $(M,M')$ of $\mathbb T$-algebras in $\mathbb E$, any natural transformation $\theta\colon Hom_{\mathbb E}(-,M)\to  Hom_{\mathbb E}(-,M')$ in $\mathcal F(\mathbb E^{op},\VV(\mathbb T))$ has an underlying natural transformation $\theta\colon  Hom_{\mathbb E}(-,M)\to  Hom_{\mathbb E}(-,M')$ in $\mathcal F(\mathbb E^{op},\mathsf{Set})$. From it, the Yoneda embedding $Y$ produces a map $f\colon M\to M'$ in $\mathbb E$; it remains to check that it is a homomorphism of $\mathbb T$-algebras, i.e. that some diagrams commute in $\mathbb E$. This, again, can be checked via the faithfulness of the Yoneda embedding.

The left exactness of the embedding $Y_{\mathbb T}$ is a consequence of the fact that the three other functors are left exact and that $U_{\mathbb T}$ reflects finite limits.
\end{proof}

So the functor  $Y_{\mathbb T}$ is left exact and conservative, and according to Lemma~\ref{conser} we get:

\begin{proposition}
	Given any finitely complete category $\EE$, if $\VV(\mathbb T)$ is a Mal'tsev variety then the category $\mathbb T(\EE)$ is Mal'tsev.
\end{proposition}

\section{Characterizations}

An immediate aspect of the richness of the notion of Mal'tsev category is that there are at least three types of characterization of very distinct nature.

\subsection{Unital characterization}

For the first one we need the following:

\begin{definition}\cite{Bourn1}
	A category  $\mathbb E$ is said to be \emph{unital} when it is pointed, finitely complete and such that, for any pair $(X,Y)$ of objects in $\mathbb E$, the following pair of monomorphisms:
	\[\xymatrix@C=2,5pc@R=2pc{{X\;}  \ar@{>->}[rr]^-{j_0^X=(1_X,0)} && X\times Y && {\; Y} \ar@{>->}[ll]_-{j_1^Y=(0,1_Y)} }\]
	is jointly extremally epic. A category  $\mathbb E$ is said to be \emph{strongly unital} when it is pointed, finitely complete and such that any reflexive relation $R$ on an object $X$ which is right punctual (i.e. containing $j_1^X$) is the largest equivalence relation on~$X$.
\end{definition}

The previous terminology is justified by the following:

\begin{lemma}\label{lemma strongly unital is unital}
	Any strongly unital category is unital.	
\end{lemma}

\begin{proof}
Consider the  following diagram
\[\xymatrix@C=3pc@R=2,5pc{ 
	{U\;} \ar@{ >->}[d]_{(f,g)}\ar@{>->}[r]^{} & T \ar@{ >->}[d]_{t}\ar[r]^{\psi} & {\; U}\ar@{ >->}[d]^{(f,g)}\\ 
	{X\times Y\;} \ar@<1ex>@{>->}[r]^{u_0\times u_1} \ar@<-1ex>[rr]_{=}	& U\times U \ar@<1ex>[r]^{f\times g} & X\times Y
}\]
where $(f,g)\colon U\into X\times Y$ is a relation containing $j_0^X$ and $j_1^Y$ through  maps $u_0$ and $u_1$ and where the right hand side square is a pullback. 
It determines a unique factorization $U\into T$ making the left hand side square a pullback as well. The relation $T$ on $U$ is defined by $(xUy)T(x'Uy')$ if and only if $xUy'$. Accordingly it is a reflexive relation. Now the map $u_1$ insures that, for all $xUy$, we have $(0U0)T(xUy)$; namely the relation $T$ is right punctual. So that we have $T=\nabla_U$ and $t$ is an isomorphism. According to the left hand side pullback, the map $(f,g)$ is itself an isomorphism, and $\EE$ is unital.
\end{proof}

The categories $\mathsf{Mon}$, $\mathsf{CoM}$ and $S\mathsf{Rg}$ of monoids, commutative monoids and semi-rings are unital categories; they are not Mal'tsev categories since, with the order $\mathbb N$ of the natural numbers, they have a reflexive relation which is not an equivalence relation. When $\mathbb E$ is finitely complete, the categories $\mathsf{Mon}(\mathbb E)$, $\mathsf{CoM}(\mathbb E)$, and $\mathsf{SRg}(\EE)$ of internal monoids, internal commutative monoids and internal semi-rings are so. This is the case in particular of the category $\mathsf{Mon}(\mathsf{Top})$ of topological monoids.
More generally, a pointed variety $\VV$ is unital if and only if it is a Jonsson-Tarski variety, see~\cite{BB}. 

\begin{lemma}
	Any pointed Mal'tsev category is strongly unital.
\end{lemma}

\begin{proof}
Given any right punctual reflexive relation $R$ on $X$, it is a right punctual equivalence relation. It follows that $R=\nabla_X$.
\end{proof}

Accordingly the categories $\mathsf{Gp}$, $\mathsf{Ab}$ and $\mathsf{Rg}$ of groups, abelian groups and rings  are strongly unital categories. When $\mathbb E$ is finitely complete, the categories $\mathsf{Gp}(\mathbb E)$, $\mathsf{Ab}(\mathbb E)$, and $\mathsf{Rg}(\EE)$ of internal groups, internal abelian groups and internal rings are so. This is the case in particular of the category $\mathsf{Gp}(\mathsf{Top})$ of topological groups.
More generally, a pointed variety of algebras $\VV$ is strongly unital if and only if it has a unique constant $0$ and a ternary operation $p$ satisfying $p(x,x,y)=y$ and $p(x,0,0)=x$, see~\cite{BB}.
Again we have:

\begin{lemma}
	Given any left exact conservative functor $U\colon \EE \to \EE'$ between finitely complete categories, the category $\EE$ is (resp. strongly) unital as soon as so is $\EE'$.
\end{lemma}

We denote by $\mathsf{Pt}(\mathbb E)$ the category whose objects are the split epimorphisms equipped with a given section and whose maps are the pairs of morphisms commuting with the split epimorphisms and the given sections:
\[\xymatrix@C=2pc@R=2pc{ X  \ar@<1.2ex>[d]^{f}\ar[r]^x & X'\ar@<1.2ex>[d]^{f'}\\
	Y\ar[u]^{s} \ar[r]_y & Y'\ar[u]^{s'}}\]
We also denote by $\pi_{\mathbb E}\colon  \mathsf{Pt}(\mathbb E)\to \mathbb E$ the functor associating with any split epimorphism $(f,s)$ its codomain $Y$. It is a fibration whose cartesian maps are the pullbacks of split epimorphisms. It is called the \emph{fibration of points} and the fibre above $Y$ is denoted by $\mathsf{Pt}_Y\mathbb E$, see~\cite{Bourn6}. We are now ready for the first characterization theorem:

\begin{theorem}\label{theounit}
	Given a finitely complete category $\EE$, the following conditions are equivalent:\\
	1) any (pointed) fibre $\mathsf{Pt}_Y\EE$ of the fibration of points is unital;\\
	2) any relation $(f,g)\colon R\rightarrowtail X\times Y$ in $\EE$ is difunctional;\\
	3) $\EE$ is a Mal'tsev category;\\
	4) any fibre $\mathsf{Pt}_Y\EE$ is strongly unital.
\end{theorem}

\begin{proof}
1) $\Rightarrow$ 2): Suppose that any  fibre $\mathsf{Pt}_Y\EE$ is unital. First, let us focus our attention on the following observation: given a pair $(R,S)$ of reflexive relations on an object $X$ such that $R\cap S=\Delta_X$, the commutative square vertically indexed by $0$ and horizontally indexed by $1$ in the following diagram is a pullback:
$$ \xymatrix@=30pt{
	R \square S \ar@<-1,ex>[d]_{p_0^R}\ar@<+1,ex>[d]^{p_1^R} \ar@<-1,ex>[r]_{p_0^S}\ar@<+1,ex>[r]^{p_1^S}
	& S \ar@<-1,ex>[d]_{d_0^S}\ar@<+1,ex>[d]^{d_1^S} \ar[l]\\
	R \ar@<-1,ex>[r]_{d_0^R} \ar@<+1,ex>[r]^{d_1^R} \ar[u]_{} & X
	\ar[u]_{} \ar[l]
}
$$
Indeed, by $R\cap S=\Delta_X$, we know that the factorization $R \square S \to R\times_XS$ is a monomorphism. Since $\mathsf{Pt}_X\EE$ is a unital category, it is an isomorphism in presence of the left hand side vertical section and of the upper horizontal one. Then the map $d_1^S\cdot p_0^S\colon R\square S \to X$ produces a connector.
Now  let be given any relation $(f,g)\colon R\to X\times Y$ in $\EE$.
By $\Eq[f] \cap \Eq[g]=\Delta_R$, we get a connector and the relation $R$ is difunctional, by Example~\ref{ddif}.\\
2) $\Rightarrow$ 3): Follows from Lemma~\ref{reflexive+difunctional}.\\
3) $\Rightarrow$ 4): When $\EE$ is a Mal'tsev category, we noticed that so is any fibre $\mathsf{Pt}_Y\EE$, which is consequently strongly unital.\\
4) $\Rightarrow$ 1): Follows from Lemma~\ref{lemma strongly unital is unital}.
\end{proof}

Thanks to Theorem~1.2.12 in~\cite{BB}, the point~2) gives rise to the following:

\begin{corollary}\label{regpush1}
	A finitely complete category $\EE$ is Mal'tsev if and only if any commutative square of split epimorphisms  (with $yf'=fx$, $xs'=sy$, $s's_y=s_xs$ and $f's_x=s_yf$)
	\[\xymatrix@C=2,5pc@R=2pc{
		X' \ar@<-1ex>@{->>}[d]_{f'} \ar@{->>}[r]^{x} 
		& X \ar@{->>}[d]_{f} \ar@<1ex>@{ >->}[l]^{s_x}\\
		Y' \ar@{ >->}[u]_{s'} \ar@{->>}[r]^{y} & Y \ar@<-1ex>@{ >->}[u]_{s} \ar@{ >->}@<1ex>[l]^{s_y} }\]
	is \emph{a regular pushout}, namely such that the factorization from $X'$ to the pullback of $f$ along $y$ is an extremal epimorphism.
\end{corollary}
\begin{remark}
As observed by Z. Janelidze \cite{janelidzeZ}, one can add one more equivalent condition to Theorem~\ref{theounit}. Recall that a \emph{subtractive category} is a finitely complete pointed category for which every left punctual reflexive relation is right punctual. A finitely complete category $\EE$ is a Mal'tsev category if and only if any fibre $\mathsf{Pt}_Y\EE$ is subtractive.
\end{remark}
\subsection{Centralization of equivalence relations}

 In this section we are going to show how this first characterization exemplifies that the Mal'tsev context is the right conceptual  one to deal with the notion of centrality of equivalence relations.

The major interest of unital categories is that it allows one to define an \emph{intrinsic notion of commutation} of morphisms. When $\mathbb E$ is a  unital category, the pair $(j_0^X,j_1^Y)$
is jointly epic; accordingly, in the following diagram, there is at most one arrow $\phi$ making the following triangles commute:
\[\xymatrix@C=2,5pc@R=2pc{ 
	{X\;} \ar[dr]_f\ar@{>->}[r]^-{j_0^X} & X \times Y \ar@{.>}[d]^{\phi} & {\; Y}\ar@{>->}[l]_-{j_1^Y}\ar[dl]^{f'}\\
	& Z\\
}\]
and the existence of such a factorization becomes a property.

\begin{definition} \cite{Bourn4, Huq}
	Let $\mathbb E$ be a unital category. We say that a pair $(f,f')$ of morphisms with common codomain \emph{commutes} (or \emph{cooperates}) when there is such a factorization map $\phi$ which is called the \emph{cooperator} of the pair.\\
	We say that the map $f\colon X\to Y$ is \emph{central} when the pair $(f,1_Y)$ cooperates and that the object $X$ is \emph{commutative} when the pair $(1_X,1_X)$ cooperates.
\end{definition}

We shall denote by $\mathsf{Com}(\EE)$ the full subcategory of commutative objects in~$\EE$. We immediately get:

\begin{proposition}\label{monab}
	Let $\mathbb E$ be a unital category. An object $X$ is commutative if and only if it is endowed with a structure of commutative internal monoid which is necessarily unique. Any morphism between commutative objects is an internal morphism of monoids.
\end{proposition}

We are going to show that the previous characterization theorem reduces the question of centralization of equivalence relations to a question of commutation in the fibres of the fibration of points. Indeed, in a Mal'tsev category, any equivalence relation $R$ on $X$ is completely determined as the following subobject in the fibre $\mathsf{Pt}_X\EE$:
\[\xymatrix@C=3pc@R=2,5pc{
	{R\;}\ar@<-1ex>[rd]_<<<<{d_0^R} \ar@{>->}[r]^-{(d_0^R,d_1^R)} & X\times X \ar@<-1ex>[d]_{p_0^X}\\
	& X  \ar[u]_{s_0^X} \ar[ul]_>>>>>>>>{s_0^R}}\]
We shall denote it by $\rho_R\colon  \Upsilon_R\into \Upsilon_{\nabla_X}$. First observe that, given any pair $(R,S)$ of equivalence relations on $X$, the product of $\Upsilon_{R^{\circ}}$ and $\Upsilon_{S}$ in this fibre coincides with the pullback introduced in Definition~\ref{connector}:
\[\xymatrix@C=2,5pc@R=2pc{
	R \times_X S\ar@<-1ex>[d]_{p_0^R} \ar[r]_-{p_1^S} 
	& S \ar[d]_{d_0^S} \ar@<-1ex>[l]_-{\sigma_1^S}\\
	R  \ar[u]_{\sigma_0^R} \ar@<-1ex>[r]_{d_1^R} & X \ar@<-1ex>[u]_{s_0^S} \ar[l]_{s_0^R} }\]
Then observe that the pair of subobjects $(\rho_{R^{\circ}},\rho_{S})$ commutes in the unital fibre $\mathsf{Pt}_X\EE$ if and only if there is a map $(d_0^S.p_1^S,p)\colon  R \times_X S \to X\times X$ in $\EE$ such that we get $p\cdot\sigma_0^R=d_0^R$ and $p\cdot\sigma_1^S=d_1^S$, namely the Mal'tsev axioms.  Accordingly there is a unique possible connector $p\colon  R\times_XS \to X $ making the pair $(R,S)$ centralize each other. So, in a Mal'tsev category, centralization of equivalence relations becomes a property; we shall denote it as usual by $[R,S]=0$.

\begin{proposition}\cite{BG,Bourn11}
	Let $\EE$ be a Mal'tsev category, and $(R,S)$ two equivalence relations on $X$. The pair $(R,S)$ centralizes each other if and only if the pair of subobjects $(\rho_{R^{\circ}},\rho_{S})$ commutes in the fibre $\mathsf{Pt}_X\EE$. This implies in particular that a pair $(R,S)$ admits at most one map $p\colon R\times_XS\to X$ satisfying the Mal'tsev axioms and that this map is necessarily a connector.
\end{proposition}

\begin{proof}
The previous observation shows  that if the pair $(R,S)$ centralizes each other in $\EE$, then the pair  $(\rho_{R^{\circ}},\rho_{S})$ commutes in the fibre $\mathsf{Pt}_X\EE$. Conversely suppose that this pair commutes; we have to show that all the axioms of Definition~\ref{connector} hold. For that, first introduce on $R\times X$ the relation $H$ defined  by $(xRy)Hz$  if and only if we have $ySz$ and $xSp(xRySz)$. For any $xRySz\in R\times_XS$ we get the following diagram relatively to the relation $H$: 
\[\xymatrix@C=2pc@R=1,5pc{
	xRy \ar@{.>}[rd] \ar[r] & y\\
	yRy \ar[r]\ar[ru] 	& z}\]
Since $\EE$ is a Mal'tsev category, the relation $H$ is a difunctional relation, and we get $(xRy)Hz$, namely $xSp(xRySz)$. We get $p(xRySz)Rz$ in the same way. 

Now define the  relation $K$ on $R\times (R\times_X S)$ by $(xRy)K(\bar yRzSw)$ by $y=\bar y$ and $p(xRySp(yRzSw))=p(xRzSw)$. We get this last identity for all $xRy$ and $yRzSw$ by the following diagram:
\[\xymatrix@C=2pc@R=1,5pc{
	xRy \ar@{.>}[rd] \ar[r] & yRzSz\\
	yRy \ar[r]\ar[ru] 	& yRzSw}\]
We get $p(p(xRySz)RzSw)=p(xRySw)$ in the same way.
\end{proof}

\begin{corollary}
	Let $\EE$ be a Mal'tsev category, and $(R,S)$ two equivalence relations on $X$. We have $[R,S]=0$ as soon as $R\cap S=\Delta_X$.
\end{corollary}

\begin{proof}
Straightforward from the first part of the proof of Theorem~\ref{theounit}.
\end{proof}

The following stability properties easily follow:

\begin{proposition}\label{stab}\cite{BG3}
	Let $\mathbb E$ be a Mal'tsev category. Let also $R,R',S$ be equivalence relations on $X$ and $\bar R, \bar S$ on $Y$. Then we get:\\
	(a) $[R,S]=0 \iff [S,R]=0$\\
	(b) $R'\subset R \; {\rm and}\;[R,S]=0\; \Rightarrow [R',S]=0$\\
	(c) $[R,S]=0 \;  {\rm and}\;[\bar R,\bar S]=0\; \Rightarrow \; [R\times \bar R,S\times \bar S]=0$\\
	(d) when $u\colon  U\into X$ is a monomorphism, we get:\\ $[R,S]=0 \Rightarrow [u^{-1}(R),u^{-1}(S)]=0$
\end{proposition}

\begin{proof}
See Propositions~3.10, 3.12 and~3.13 in~\cite{BG3}. 
\end{proof} 

\begin{definition}
	Let $\mathbb E$ be a Mal'tsev category. An equivalence relation $R$ on an object $X$ is said to be \emph{abelian} when we have $[R,R]=0$ and \emph{central} when we have $[R,\nabla_X]=0$. An object $X$ is said to be \emph{affine} when we have $[\nabla_X,\nabla_X]=0$.	
\end{definition}

We shall denote by $\mathsf{Aff}(\EE)$ the full subcategory of affine objects in $\EE$.

\begin{proposition}\label{maladd}
 Let $\mathbb E$ be a Mal'tsev category. When an equivalence relation $R$ on $X$ is abelian, then the connector $p$ realizing $[R,R]=0$ is such that $p(xRyRz)=p(zRyRx)$. An object $X$ is affine if and only if it is endowed with a (necessarily unique) internal Mal'tsev operation which is necessarily associative and commutative. Any morphism between affine objects commutes with the internal Mal'tsev operations.
\end{proposition}

\begin{proof}
Define on $R$ the relation $L$ define by $(x,y)L(z,w)$ if and only if we have $y=z$ and $p(xRyRw)=p(wRyRx)$. The following diagram holds when $R$ is abelian and determines our assertion:
\[\xymatrix@C=2pc@R=1,5pc{
	xRy \ar@{.>}[rd] \ar[r] & yRy\\
	yRy \ar[r]\ar[ru] 	& yRz}\]
The next point is just the description of the connector $p$ associated with the centralization $[\nabla_X,\nabla_X]=0$.
\end{proof}

\begin{corollary}\label{aff}
Let $\mathbb E$ be a Mal'tsev category. $\mathsf{Aff}(\EE)$ is stable under finite products and subobjects in $\EE$. An internal abelian group in $\EE$ is just a pointed affine object $0_A\colon  1\into A$.
\end{corollary}

\begin{proof}
It is a direct consequence of points~(c) and (d) in Proposition~\ref{stab} and of Proposition~\ref{maladd}.
\end{proof}

\subsection{Groupoid characterization}

Let us denote respectively by  $\mathsf{RG}(\EE)$ and  $\mathsf{Grpd}(\EE)$ the categories of internal reflexive graphs and of internal groupoids in any finitely complete category $\EE$. The forgetful functor $W_{\EE}\colon \mathsf{Grpd}(\EE) \to \mathsf{RG}(\EE)$ is left exact and conservative. As such, it is faithful and any monomorphism of $\mathsf{Grpd}(\EE)$ is hypercartesian with respect to~$W_{\EE}$.

\noindent\textbf{Internal groupoids}

According to Example~\ref{exGroupoid} following which a groupoid is a reflexive graph
endowed with a connector $p$ on the pair $(\Eq[d_0], \Eq[d_1])$, we immediately get the first part of the following:

\begin{lemma}\label{subreflexive}
Given a Mal'tsev category $\EE$, there is on any reflexive graph at most one groupoid structure. Moreover, the induced inclusion functor $W_{\EE}\colon  \mathsf{Grpd}(\EE) \into \mathsf{RG}(\EE)$ is full~\cite{CPP} and such that any sub-reflexive graph of a groupoid is itself a groupoid~\cite{Bourn1}. 
\end{lemma}

\begin{proof}
Let be given a morphism of reflexive graphs:
$$ \xymatrix@=17pt{
	X_1\ar@<-1,ex>[d]_{d_0}\ar@<+1,ex>[d]^{d_1} \ar[rr]^{f_1}
	&& Y_1 \ar@<-1,ex>[d]_{d_0}\ar@<+1,ex>[d]^{d_1}\\
	X_0  \ar[rr]_{f_0} \ar[u]_{}& & Y_0 \ar[u]_{}   
}
$$
When these reflexive graphs are underlying groupoid structures, the commutation of this morphism with the connectors is checked by composition with the extremally epic pair involved in the definition of $\Eq[d_0]\times _{X_0} \Eq[d_1]$.

Now given any subobject $(f_0,f_1)\colon \mathbb X \into \mathbb Y$ in $\mathsf{RG}(\EE)$ with $\mathbb Y$ a groupoid, the inverse image $f_0^{-1}(\mathbb Y)$ along $f_0$ determines a subobject:
$$ \xymatrix@=17pt{
	{X_1\;}\ar@<-1,ex>[d]_{d_0}\ar@<+1,ex>[d]^{d_1} \ar@{>->}[rr]^{\phi_1}
	&& f_0^{-1}(\mathbb Y) \ar@<-1,ex>[d]_{d'_0}\ar@<+1,ex>[d]^{d'_1}\\
	X_0  \ar@{=}[rr] \ar[u]_{}& & X_0 \ar[u]_{}   
}
$$ 
where the right hand side part is a groupoid as well. Now $\Eq[d_i]=\phi_1^{-1}(\Eq[d'_i])$ so that $\mathbb X$ is underlying a groupoid structure by the point~d) in Proposition~\ref{stab}.
\end{proof}

Whence another characterization theorem:

\begin{theorem}\label{chargr}\cite{Bourn1} 
Given any finitely complete category $\EE$, the following conditions are equivalent:\\
1) $\EE$ is a Mal'tsev category;\\
2) the forgetful functor $W_{\mathbb E}\colon \mathsf{Grpd}(\EE) \to \mathsf{RG}(\EE)$ is saturated on subobjects, namely any subobject $n\colon \mathbb X \into W_{\EE}(\mathbb Y)$ in $\mathsf{RG}(\EE)$ is the image, up to isomorphism, of a monomorphism $m\colon  \mathbb{\overline{X}} \into \mathbb Y$ in $\mathsf{Grpd}(\EE)$.
\end{theorem}

\begin{proof}
$[1) \Rightarrow 2)]$ is a direct consequence of the previous lemma. Suppose 2) and start with a reflexive relation $R$ on $X$. Then condition 2) applied to the inclusion $R\into \nabla_X$ in $\mathsf{RG}(\EE)$ makes $R$ an equivalence relation. 
\end{proof}

\noindent\textbf{Internal categories}

Now, what are the internal categories in a Mal'tsev category? The answer is given by the following result which provides us with another characterization of internal groupoids.

\begin{proposition}\label{intcat}\cite{CPP,Bourn11}
	Let $\EE$ be a Mal'tsev category and
	$\xymatrix@C=1cm{ X_1 \ar@<1.2ex>[r]^{d_0}\ar@<-1.2ex>[r]_{d_1} & X_0 \ar[l]|(.50){s_0} }$
	an internal reflexive graph. The following conditions are equivalent:\\
	1) the following subobjects commute in the fibre $\mathsf{Pt}_{X_0}\EE$:
	$$
	\xymatrix@=30pt{
		{X_1\;} \ar@{>->}[r]^-{(d_0,1_{X_{1}})}\ar@<-1ex>[rd]_{d_0} & X_0\times X_1 \ar@<-1ex>[d]_{p_{X_0}} & {\;X_1} \ar@{>->}[l]_-{( d_1,1_{X_{1}})} \ar@<2ex>[ld]^{d_1}\\
		& X_0 \ar[lu]_>>>>>>>{s_0} \ar[u]_-{(1_{X_{0}},s_0)} \ar@<-1ex>[ru]^>>>>>>{s_0}
	}
	$$
	2) this reflexive graph is underlying an internal category; \\
	3) this reflexive graph is underlying an internal groupoid.
\end{proposition}

\begin{proof}
The two subobjects commute in $\mathsf{Pt}_{X_0}\EE$ if and only if they have a cooperator $\phi\colon X_1\times_{X_0} X_1\rightarrow X_0\times X_1$, i.e.\ a morphism satisfying $\phi \cdot s_0=(d_1,1_{X_{1}})$, $\phi \cdot s_1=(d_0,1_{X_{1}})$, $p_{X_0}\cdot \phi = d_0d_2$ and $\phi \cdot s_1s_0=(1_{X_0},s_0)$:
$$
\xymatrix@=30pt{
	& X_1\times_{X_0} X_1 \ar[d]_{\phi} \ar@{}[dl]|(.25){}="A" \ar@<-1.5ex> "A";[dl]_-{d_2} \ar@{}[dr]^(.25){}="B" \ar@<1.5ex> "B";[dr]^{d_0}\\
	{X_1\;} \ar@{>->}[r]^-{(d_0,1_{X_{1}})} \ar@<-2ex>[rd]_-{d_0} \ar@<.5ex>[ur]_>>>>>>{s_1} & X_0\times X_1 \ar@<-1ex>[d]_-{p_{X_0}} & {\;X_1} \ar@{>->}[l]_-{(d_1,1_{X_{1}})} \ar@<2ex>[ld]^-{d_1} \ar@<-.5ex>[ul]^>>>>>>{s_0}\\
	& X_0 \ar@<1ex>[lu]_>>>>>>>>{s_0} \ar[u]_{(1_{X_{0}},s_0)} \ar@<-1ex>[ru]^>>>>>>>{s_0}
}
$$
where the whole quadrangle is the pullback which defines the internal object $X_1\times_{X_0} X_1$ of composable pairs of the reflexive graph. So the morphism $\phi$ is necessarily a pair of the form $(d_0d_2,d_1)$, where $d_1\colon X_1\times_{X_0} X_1\rightarrow X_1$ is such that $d_1s_0=1_{X_1},\; d_1s_1=1_{X_1}$. The incidence axioms: $d_0d_1=d_0d_0$, $d_1d_1=d_1d_2$ come by composition with the upper jointly extremally epic pair $(s_0,s_1)$. Accordingly this map $d_1$ produces a composition for the composable  pairs. Let us set $X_2=X_1\times_{X_0} X_1$. In order to check the associativity we need the following pullback which defines $X_3$ as the internal objects of `triples of composable morphisms':
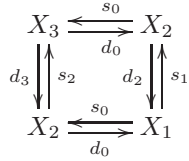
$$\xymatrix{
	X_3  \ar@<-2pt>[r]_(.6){d_0} \ar@<-2pt>[d]_{d_3} & X_2  \ar@<-2pt>[l]_(.4){s_0} \ar@<-2pt>[d]_{d_2} \\
	X_2   \ar@<-2pt>[r]_{d_0} \ar@<-2pt>[u]_{s_2} & X_1 \ar@<-2pt>[l]_{s_0} \ar@<-2pt>[u]_{s_1} }
$$
The composition map $d_1$ induces a unique couple of maps $(d_1,d_2)\colon X_3 \rightrightarrows X_2$ such that
$d_0d_1=d_0d_0$, $d_2d_1=d_1d_3$ and $d_0d_2=d_1d_0$, $d_2d_2=d_2d_3$. The associativity axiom is given by the remaining simplicial axiom: 
(3) $d_1d_1=d_1d_2$.
The checking of this axiom comes with composition with the pair $(s_0,s_2)$ of the previous diagram since it is jointly extremally epic as well.

Conversely, the composition morphism $d_1\colon X_1\times_{X_0} X_1\rightarrow X_1$ of an internal category satisfies $d_1s_0=1_{X_1},\; d_1s_1=1_{X_1}$ and consequently produces the cooperator $\phi=(d_0d_2,d_1)$. Whence 1) $\Leftrightarrow$ 2).

It is clear that $3)\Rightarrow 2)$. It remains to check $2)\Rightarrow 3)$. Starting with an internal category, consider the following diagram in the fibre $\mathsf{Pt}_{X_0}\EE$:
$$
\xymatrix{
	X_1\times_{X_0} X_1 \ar[r]^{(d_0,d_1)} \ar@<-1ex>[rd]_{d_0} & \Eq[d_0] \ar@<-4pt>[r]_(.6){d_1} \ar@<-4pt>[d]_{d_0} & X_1  \ar@<-4pt>[l]_(.4){s_1} \ar@<-4pt>[d]_{d_0} \ar@(u,u)[ll]_{s_1} \\
	&	X_1   \ar@<-4pt>[r]_{d_0} \ar@<-4pt>[u]_{s_0} \ar[ul]_{s_0} & X_0 \ar@<-4pt>[l]_{s_0} \ar@<-4pt>[u]_{s_0} }
$$
First let us show that $(d_0,d_1)$ is a monomorphism, namely that the composition is right cancelable. So, consider the relation $H$ on $(\Eq[d_0]\cap \Eq[d_1])\times X_1$ (where $\Eq[d_0]\cap \Eq[d_1]$ is the object of `parallel maps' in the internal category) defined  by $(\alpha,\beta)H\gamma$ if $d_0(\alpha)=d_1(\gamma)$ and $\alpha\cdot\gamma=\beta\cdot\gamma$. The following diagram shows that $(\alpha,\beta)H1_{d_1(\gamma)}$, namely that $\alpha=\beta$, as soon as $\alpha\cdot\gamma=\beta\cdot\gamma$:
\[\xymatrix@C=2pc@R=1,5pc{
	(\alpha,\beta) \ar@{.>}[rd] \ar[r] & \gamma\\
	(1_{d_1(\gamma)},1_{d_1(\gamma)}) \ar[r]\ar[ru] 	& 1_{d_1(\gamma)}}\]
Accordingly $(d_0,d_1)\colon X_1\times_{X_0} X_1 \into \Eq[d_0]$ produces a reflexive (due to the $s_0$) and right punctual  (commutation of the $s_1$) relation on the object $(d_0,s_0)\colon X_1\splito X_0$ in the strongly unital fibre $\mathsf{Pt}_{X_0}\EE$. Accordingly it is an isomorphism and the internal category is an internal groupoid.
\end{proof}

\subsection{Base-change characterization}

The next characterization is dealing with the base-change functor along split epimorphisms with respect to the fibration of points.

\begin{definition}
Given a split epimorphic pair of functors $(T,G)\colon  \EE \splito \EE'$, $(T\circ G=1_{\EE'})$, this pair is called \emph{correlated} (resp. \emph{strongly correlated}) on monomorphisms when, given any monomorphism $m\colon Z \into G(X)$ in $\EE$, the morphism $GT(m)\colon GT(Z)\to G(X)$ factorizes through $m$ (resp. factorizes through $m$ via an isomorphism).
\end{definition}

\begin{lemma}
	Given a split epimorphic  pair of functors $(T,G)$, if it is correlated on monomorphisms, then  any monomorphism $m\colon  Z\into G(X)$ such that $T(m)$ is an isomorphism is itself an isomorphism. When $\EE$ is finitely complete and $T$ is left exact, we have the converse.
\end{lemma}

\begin{proof}
Suppose the pair is correlated and $T(m)$ an isomorphism. Then, so is $GT(m)$ and the monomorphism $m$ is a split epimorphism as well. Accordingly it is an isomorphism.
When $T$ is left exact and $m\colon Z \into G(X)$ a monomorphism, the map $T(m)$ is a monomorphism as well. Now consider the following pullback in $\EE$:
\[\xymatrix@C=2pc@R=1,5pc{
	P \ar@{ >->}[d]_{\bar m} \ar[r]^{} & Z\ar@{ >->}[d]^{m} \\
	GT(Z) \ar[r]_{GT(m)} & G(X)}\]
It is preserved by $T$ so that $T(\bar m)$ is an isomorphism. Under our assumption, so is $\bar m$; and $GT(m)$ factorizes through $m$. 
\end{proof}

If $f\colon Y \to X$ is a morphism in a finitely complete category~$\EE$, we denote by $f^*\colon \mathsf{Pt}_X\EE \to \mathsf{Pt}_Y\EE$ the base-change functor obtained by pullback along $f$. If $\EE$ is pointed, we denote by $\alpha_X$ (resp. $\tau_X$) the unique morphism $0\to X$ (resp. $X \to 0$) for any object $X$. From the above lemma, it is easy to check that, \emph{given a finitely complete pointed category $\EE$, it is unital if and only if for any object $X$ the split epimorphic pair of base-change functors $(\alpha_X^*,\tau_X^*)\colon  \mathsf{Pt}_X\EE \splito \mathsf{Pt}_0\EE \cong \EE$ is correlated on monomorphisms}. A bit more difficult is the same characterization dealing with strongly unital categories  and strongly correlated pairs $(\alpha_X^*,\tau_X^*)$; for that see~\cite{Bourn1} and the second assertion of the following:

\begin{proposition}
Let $(T,G)\colon  \EE \splito \EE'$ be a split epimorphic pair of functors with $\EE$ finitely complete and $T$ left exact. If it is correlated on monomorphisms, then the faithful functor $G$ is full as well. It is strongly correlated if and only if the functor $G$ is saturated on subobjects.
\end{proposition}

\begin{proof}
Suppose we have a map $f\colon G(U)\to G(V)$, then take the equalizer $j\colon  W\into G(U)$ of the pair $(f,GT(f))$. The functor $T$ being left exact, its image $T(j)$ is an isomorphism. Accordingly $j$ is itself an isomorphism, and $f=GT(f)$. So, $G$ is full. When the pair $(T,G)$ is strongly correlated and $m\colon Z \into G(X)$ is a monomorphism, then $n=T(m)$ is the monomorphism whose image by $G$ is isomorphic to $m$. Conversely, if $G$ is saturated on subobjects, starting with a monomorphism $m\colon Z \into G(X)$, denote by $\gamma\colon  Z\to G(W)$ the isomorphism such that $m\cong G(n)$ for a monomorphism $n\colon W \into X$. Then the map $\gamma^{-1}\cdot GT(\gamma)\colon GT(Z)\to Z$ produces the desired isomorphism  making the pair $(T,G)$ strongly correlated.
\end{proof}

Whence, now, the third characterization:

\begin{theorem}\label{theo34}\cite{Bourn1}
	Let $\EE$ be a finitely complete category. It is Mal'tsev if and only if, given any split epimorphism $(f,s)$ in $\EE$, the (left exact) base-change functor $f^*$ with respect to the fibration of points is saturated on subobjects (and consequently full).
\end{theorem}

\begin{proof}
By Theorem~\ref{theounit}, the category $\EE$ is Mal'tsev if and only if any fibre $\mathsf{Pt}_Y\EE$ is strongly unital, which is the case if and only if, for any split epimorphism $(f,s)$, the pair $(s^*,f^*)$ is strongly correlated. According to the last proposition, it is the case if and only if $f^*$ is saturated on subobjects. 
\end{proof}

This last characterization is also important because, in the regular context, we shall show that it could be extended to any regular epimorphism $f$ in $\EE$ (see Theorem~\ref{theorem change of base saturated regular}).

\section{Stiffly and naturally Mal'tsev categories}\label{section Stiffly and naturally Mal'tsev categories}

There are obviously two extremal situations satisfied by Mal'tsev categories: any equivalence relations $R$ and $S$ on a same object centralize each other and the only pairs $(R,S)$ of equivalence relations centralizing each other are the ones such that $R\cap S=\Delta_X$.

The following notion was introduced by P. Johnstone in~\cite{Johnstone} (more precisely, he defined this notion via the equivalent formulation~\emph{2} in Theorem~\ref{equivalent-nat} below):
\begin{definition}\cite{Johnstone}
A Mal'tsev category is \emph{naturally Mal'tsev} if $[R,S]=0$ for any equivalence relations $R$ and $S$ on a same object.
\end{definition}

Any finitely complete additive category is naturally Mal'tsev. So is the subcategory $\mathsf{Aff}(\EE)$ of any Mal'tsev category $\EE$.

\begin{definition}
We say that a Mal'tsev category is \emph{stiffly Mal'tsev} if for any equivalence relations $R$ and $S$ on a same object, $[R,S]=0$ if and only if $R\cap S=\Delta_X$.
\end{definition}

The categories $\mathsf{BoRg}$ of boolean rings and $\mathsf{CNRg}$ of commutative von Neumann regular rings are examples of stiffly Mal'tsev categories.
The variety $\mathsf{Heyt}$ of Heyting algebras being a stiffly Mal'tsev category, so are the dual $\mathsf{Set}^{op}$ of the category of sets, and more generally the dual $\EE^{op}$ of any elementary topos $\EE$ in view of the left exact conservative functor $\EE^{op} \to \mathsf{Heyt}(\EE)$~\cite{BB}.

Clearly the two notions are stable under slicing and coslicing. For the next respective characterizations, we need the following:

\begin{definition}
A \emph{linear category} is a unital category where any object is commutative. A \emph{stiffly unital category} is a unital category in which the $0$ object is the unique commutative object.
\end{definition}

When $\EE$ is unital, the subcategory $\mathsf{Com}(\EE)$ is linear. The category $\mathsf{BoSRg}$ of boolean semi-rings is stiffly unital.

\begin{theorem}\cite{BB}
	Given any Mal'tsev category $\EE$, the following statements are equivalent:\\
	1) $\EE$ is a stiffly Mal'tsev category;\\
	2) the only internal groupoids are the equivalence relations;\\
	3) any fibre $\mathsf{Pt}_Y\EE$ is stiffly unital;\\
	4) the only abelian equivalence relations are the discrete ones.
\end{theorem}

\begin{proof}
1) $\Rightarrow$ 2): Consider any internal groupoid:
$\xymatrix@C=1cm{ X_1 \ar@<1.2ex>[r]^{d_0}\ar@<-1.2ex>[r]_{d_1} & X_0 \ar[l]|(.50){s_0} }$.
We have $[\Eq[d_0],\Eq[d_1]]=0$ and thus $\Eq[d_0]\cap \Eq[d_1]=\Delta_{X_1}$. So, $(d_0,d_1)\colon X_1\into X_0\times X_0$ is a monomorphism, and the groupoid is an equivalence relation.\\
2) $\Rightarrow$ 3): A split epimorphism
$(f,s)\colon  X\splito Y$  is a commutative object in $\mathsf{Pt}_Y\EE$, when it is endowed with a monoid structure in this fibre, namely when it is an internal category structure with $d_0=f=d_1$. According to Proposition~\ref{intcat} it is a groupoid. So $(f,f)$ and thus $f$ are monomorphisms. Being a split epimorphism as well, $f$ is an isomorphism, i.e.\ a $0$ object in $\mathsf{Pt}_Y\EE$.\\
3) $\Rightarrow$ 4): Given any equivalence relation $R$ on $X$, we have $[R,R]=0$ if and only if the split epimorphism $(d_0^R,s_0^R)\colon  R\splito X$ is endowed with a commutative monoid structure in the fibre $\mathsf{Pt}_X\EE$. So, $d_0^R$ is an isomorphism and we get $R\cong \Delta_X$.\\
4) $\Rightarrow$ 1): If we have $[R,S]=0$, we get $[R\cap S,R\cap S]=0$; so $R\cap S=\Delta_X$ and we get~1).
\end{proof}

\begin{theorem}\label{equivalent-nat}\cite{Johnstone,Bourn1}
	Given any Mal'tsev category $\EE$, the following statements are equivalent:\\
	1) $\EE$ is a naturally Mal'tsev category;\\
	2) any object $X$ is endowed with a natural Mal'tsev operation $p_X$;\\
	3) any fibre $\mathsf{Pt}_Y\EE$ is linear;\\
	4) any fibre $\mathsf{Pt}_Y\EE$ is additive;\\
	5) any reflexive graph is endowed with a groupoid structure.
\end{theorem}

\begin{proof}
We have [1) $\Leftrightarrow$ 2)] with the connector $p_X\colon X\times X \times X \to X$ associated with the centralization $[\nabla_X,\nabla_X]=0$, see Proposition~\ref{maladd}.
Now, $[1)\Rightarrow 5)]$ is straightforward. We get $[5)\Rightarrow 4)]$ since any split epimorphism $(f,s)\colon X\splito Y$ is a particular reflexive graph, which, by~5), gives to $(f,s)$ an internal group structure in the fibre $\mathsf{Pt}_Y\EE$. Now $[4)\Rightarrow 3)]$ is a consequence of the fact that any finitely complete additive category is  linear.
Suppose~3) and consider the split epimorphism $(p_0^Y,s_0^Y)\colon Y\times Y\to Y$. It has a monoid structure in the fibre $\mathsf{Pt}_Y\EE$, which is a ternary operation $p_Y\colon Y\times Y\times  Y \to Y$ in $\EE$. It satisfies the unit axioms which, for $p_Y$, turn to be exactly the two Mal'tsev axioms in~$\EE$.
\end{proof}

\section{Regular Mal'tsev categories}\label{section regular Maltsev categories}

An equivalence relation is called \emph{effective} when it is the kernel pair of some morphism. A map $f$ is said to be a \emph{regular epimorphism} when $f$ is the coequalizer of two parallel arrows in $\EE$. Recall from~\cite{Barr}:

\begin{definition}
	A finitely complete category $\mathbb E$ is \emph{regular} when:\\
	(a) regular epimorphisms are stable under pullbacks;\\
	(b) any effective equivalence relation $\Eq[f]$ has a coequalizer.\\
	It is \emph{exact} when, in addition: \\
	(c) any equivalence relation in $\EE$ is effective. 
\end{definition}

Any variety $\VV$ of universal algebras is an exact category. Given any map $f$ in a regular category the quotient $q_f$ of the kernel equivalence relation $\Eq[f]$ produces a canonical decomposition $f=m\cdot q_f$ where $m$ is a monomorphism~\cite{Barr}. Given any regular epimorphism $g\colon X\onto Y$ and any equivalence relation $R$ on~$X$, the canonical decomposition of the map $R\into X\times X \stackrel{g\times g}{\onto} Y\times Y$ produces a reflexive relation $S$ on $Y$. When $\EE$ is a Mal'tsev category, it is an equivalence relation we shall denote by $g(R)$. Now we have the following:

\begin{proposition}\label{regstab}\cite{BG3}
	Let $\EE$ be a regular Mal'tsev category and $g\colon X\onto Y$ a regular epimorphism. If $(R,S)$ is a pair of centralizing equivalence relations on $X$, then the equivalence relations $g(R)$ and $g(S)$ also centralize each other. In particular, when $X$ is an affine object, so is $Y$.
\end{proposition}

When, moreover, the category $\EE$ is finitely cocomplete, we can produce the commutator of any pair of equivalence relations:

\begin{proposition}\cite{Bourn12, BB}
	Let $\EE$ be a finitely cocomplete regular Mal'tsev category. If $(R,S)$ is a pair of equivalence relations on $X$, there is a universal regular epimorphism $\psi\colon X\onto Y$ such that we get $[ \psi(R),\psi(S)]=0$. In particular the inclusion $\mathsf{Aff}(\EE)\into \EE$ has a left adjoint.
\end{proposition}

\begin{remark}
Given two equivalence relations $R$ and $S$ on a same object in a finitely cocomplete regular Mal'tsev category, the \emph{commutator} $[R,S]$ of $R$ and $S$ can be defined as the kernel equivalence relation $Eq[\psi]$ of the morphism $\psi\colon X\onto Y$ from the above proposition.
\end{remark}

\begin{remark}
The Mal'tsev context being the right conceptual one to deal with the notion of centrality of equivalence relations, it is not unexpected to observe that it is also the right context to deal with nilpotency as well~\cite{BergerB}.
\end{remark}

In the regular context we get the following observations and characterization:

\begin{lemma}\label{transal}
	Let $\EE$ be a regular Mal'tsev category and the following diagram be any pullback of a split epimorphism $f$ along a regular epimorphism $q$:
	$$
	\xymatrix{
		X'  \ar@{->>}[r]^{q'} \ar@<-2pt>[d]_{f'} & X  \ar@<-2pt>[d]_f\\
		Y'   \ar@{->>}[r]_q \ar@<-2pt>[u]_{s'} & Y  \ar@<-2pt>[u]_s }
	$$
	Then the upward square is a pushout.
\end{lemma}

\begin{proof}
Consider any pair $(\phi,\sigma)$ of morphisms such that $\phi \cdot s'=\sigma \cdot q\;(*)$:
$$\xymatrix@=6pt{
	\Eq[q'] \ar@<-1ex>[ddd]_{\Eq(f')} \ar@<-1ex>[rrrr]_<<<<<<<<<{d_{1}^{q'} }\ar@<1ex>[rrrr]^<<<<<<<<<{d_{0}^{q'}} &&&&  {X'\;}  \ar[llll]  \ar@<-1ex>[ddd]_{f'}  \ar@{->>}[rrrr]^{q'}\ar[rrrrrd]_{\phi}  &&&&  {X\;}  \ar@<-1ex>[ddd]_{f}  \\
	&   &&&&   & & &  & T   && \\
	&&& &&&&&&\\
	\Eq[q] \ar[uuu]_{\Eq(s')} \ar@<-1ex>[rrrr]_{d_{1}^q} \ar@<1ex>[rrrr]^{d_{0}^q}  &&&&  Y' \ar[llll] \ar@{->>}[rrrr]_{q}  \ar[uuu]_{s'}
	&& && Y   \ar@<-1ex>[uur]_{\sigma} \ar[uuu]_{s}
}
$$
and complete the diagram by the kernel pairs $\Eq[q]$ and $\Eq[q']$ which produce the left hand side pullbacks. The morphism $q'$ being a regular epimorphism, it is the quotient of its kernel pair $\Eq[q']$. We shall obtain the desired factorization $X\to T$ by showing that $\phi$ coequalizes the pair $(d_0^{q'},d_1^{q'})$. The left hand side squares being pullbacks, this can be done by composition with the jointly extremal pair $(\Eq(s'),s_0^{q'}), \;{\rm with} \; s_0^{q'}\colon X'\into \Eq[q']$ the diagonal giving the reflexivity of $\Eq[q']$. This is trivial for the composition with $s_0^{q'}$, and a consequence of the equality $(*)$ for the composition with $\Eq(s')$.
\end{proof}

\begin{theorem}\label{theorem change of base saturated regular}\cite{Bourn1}
	Given any regular category $\EE$, it is a Mal'tsev category if and only if any base-change functor $q^*$ with respect to the fibration of points along a regular epimorphism $q$ is fully faithful and saturated on subobjects. 	
\end{theorem}

\begin{proof}
Any split epimorphism being a regular one, the condition above implies that $\EE$ is a Mal'tsev category thanks to Theorem~\ref{theo34}. Let us show the converse. First notice that in any regular category, and  given any regular epimorphism $q$, the base-change $q^*$ is necessarily faithful. Suppose, in addition, that $\EE$ is a Mal'tsev category.\\
1) \emph{The functor $q^*$ is full}. Consider the following diagram:
$$\xymatrix@=9pt{
	{X'\;} \ar[rd]^{m'}   \ar@<-1ex>[ddd]_<<<<{f'}  \ar@{->>}[rrrr]^{q'}  &&&&  {X\;}  \ar@<-1ex>[ddd]_<<<<{f}  \\
	&   \bar X' \ar[ddl]_{\bar f'} \ar@{->>}[rrrr]^<<<<{\bar q}  & & &  & \bar X \ar[ddl]_{\bar f}  && \\
	&&& &&&&&&\\
	Y'  \ar@{->>}[rrrr]_{q}  \ar[uuu]_>>>>{s'} \ar@<-1ex>[uur]_{\bar s'} && && Y  \ar@<-1ex>[uur]_{\bar s} \ar[uuu]_>>>>{s}
}
$$
where the downward squares are pullbacks and $m'$ a morphism in $\mathsf{Pt}_{Y'}\EE$. According to the previous lemma the upward vertical square is a pushout; whence  a unique map $m\colon X\rightarrow \bar X$ such that $m\cdot q'=\bar q\cdot m'$ and $m\cdot s=\bar s$; we get also $\bar f\cdot m=f$ since $q'$ is a regular epimorphism, and $m$ is a map in the fibre $\mathsf{Pt}_Y\EE$ such that $q^*(m)=m'$.

\smallskip

2) \emph{The functor $q^*$ is saturated on subobjects}. First, any base-change functor $g^*$, being left exact, preserves monomorphisms. Consider now the following diagram where the right hand side quadrangle is a pullback and $m$ is a monomorphism in $\mathsf{Pt}_{Y}\EE$:
$$\xymatrix@=10pt{
	\Eq[\bar q\cdot m] \ar@<-1ex>[ddd]_{\Eq(f')} \ar@{>->}[rd]_{\Eq(m)} \ar@<-1ex>[rrrr]_<<<<<<<<{\delta_{1}}\ar@<1ex>[rrrr]^<<<<<<<<{\delta_{0}} &&&&  {X'\;} \ar@{>->}[rd]^{m} \ar[llll]  \ar@<-1ex>[ddd]_<<<<{f'}   &&&&   \\
	&  \Eq[\bar q] \ar[ddl] \ar@<-1ex>[rrrr]_<<<<{d_{1}^{\bar q}} \ar@<1ex>[rrrr]^<<<<{d_{0}^{\bar q}} &&&&  \bar X' \ar[llll]
	\ar[ddl]_{\bar f'} \ar@{->>}[rrrr]^-{\bar q}  & & &  & \bar X \ar[ddl]_{\bar f}  && \\
	&&& &&&&&&\\
	\Eq[q] \ar[uuu]_{\Eq(s')} \ar@<-1ex>[rrrr]_{d_{1}^q} \ar@<1ex>[rrrr]^{d_{0}^q} \ar@<-1ex>[uur] &&&&  Y' \ar[llll] \ar@{->>}[rrrr]_{q}  \ar[uuu]_>>>>{s'}
	\ar@<-1ex>[uur]_{\bar s'} && && Y   \ar@<-1ex>[uur]_{\bar s}
}
$$
Complete the diagram with the kernel pair $\Eq[\bar q\cdot m]$. The factorization $\Eq(m)$ is a monomorphism. In the Mal'tsev context, this implies that any of the left hand side commutative squares is a pullback: indeed, since $\Eq(m)$ is a monomorphism, it is also the case for the factorization $\tau$ of the left hand side square indexed by $0$ to the pullback of $(f',s')$ along the split epimorphism $(d_0^q,s_0^q)$; but it is an extremal epimorphism as well, since $\EE$ is a Mal'tsev category, by Condition 1 in Theorem~\ref{theounit}; so it is an isomorphism.

So, the following downward left hand side diagram is underlying a discrete fibration between equivalence relations.
Now, denote by $q'$ the quotient of the effective relation $\Eq[\bar q\cdot m]$, and by $(f,s)$ the induced split epimorphism.
$$\xymatrix@=7pt{
	\Eq[\bar q\cdot m] \ar@<-1ex>[ddd]_{\Eq(f')}  \ar@<-1ex>[rrrr]_<<<<<<<<{\delta_{1}}\ar@<1ex>[rrrr]^<<<<<<<<{\delta_{0}} &&&&  {X'\;} \ar[llll]  \ar@<-1ex>[ddd]_{f'}  \ar@{.>>}[rrrr]^{q'} &&&&  {X\;}  \ar@<-1ex>[ddd]_{f} \\
	&  &&&&   & & &  &  && \\
	&&& &&&&&&\\
	\Eq[q] \ar[uuu]_{\Eq(s')} \ar@<-1ex>[rrrr]_{d_{1}^q} \ar@<1ex>[rrrr]^{d_{0}^q}    &&&&  Y'  \ar[llll] \ar@{->>}[rrrr]_{q}  \ar[uuu]_{s'}
	&& && Y \ar[uuu]_{s} 
}
$$
By the so-called Barr-Kock Theorem~\cite{Barr, BG2}, the right hand side square is a  pullback in the regular category $\EE$. Since $q^*$ is full, $m$ determines a factorization $n\colon X\to \bar X$ in the fibre $\mathsf{Pt}_Y\EE$:
$$\xymatrix@=10pt{
	\Eq[\bar q\cdot m] \ar@<-1ex>[ddd]_{\Eq(f')} \ar@{>->}[rd]_{\Eq(m)} \ar@<-1ex>[rrrr]_<<<<<<<<{\delta_{1} }\ar@<1ex>[rrrr]^<<<<<<<<{\delta_{0}} &&&&  {X'\;} \ar@{>->}[rd]^{m} \ar[llll]  \ar@<-1ex>[ddd]_<<<<{f'}  \ar@{->>}[rrrr]^{q'} &&&&  {X\;}  \ar@<-1ex>[ddd]_{f}  \ar[rd]^{n}\\
	&  \Eq[\bar q] \ar[ddl] \ar@<-1ex>[rrrr]_<<<<{d_{1}^{\bar q}} \ar@<1ex>[rrrr]^<<<<{d_{0}^{\bar q}} &&&&  \bar X' \ar[llll]
	\ar[ddl]_{\bar f'} \ar@{->>}[rrrr]^<<<<{\bar q}  & & &  & \bar X \ar[ddl]_{\bar f}  && \\
	&&& &&&&&&\\
	\Eq[q] \ar[uuu]_{\Eq(s')} \ar@<-1ex>[rrrr]_{d_{1}^q} \ar@<1ex>[rrrr]^{d_{0}^q} \ar@<-1ex>[uur] &&&&  Y' \ar[llll] \ar@{->>}[rrrr]_{q}  \ar[uuu]_>>>>{s'}
	\ar@<-1ex>[uur]_{\bar s'} && && Y \ar[uuu]_{s}  \ar@<-1ex>[uur]_{\bar s}
}
$$
The upper right hand side quadrangle is a pullback since the two other right hand side commutative squares are so. Accordingly we get $m=q^*(n)$ and $n$ is a monomorphism since pulling back along regular epimorphisms reflects monomorphisms~\cite{Barr}. 
\end{proof}

From Corollary~\ref{regpush1}, we get another characterization:

\begin{corollary}\label{regpush2}
	A regular category $\EE$ is Mal'tsev if and only if any morphism in $\mathsf{Pt}(\EE)$ with horizontal regular epimorphisms
	\[\xymatrix@C=2,5pc@R=2pc{
		X' \ar@<-1ex>@{->>}[d]_{f'} \ar@{->>}[r]^{x} 
		& X \ar@{->>}[d]_{f} \\
		Y' \ar@{ >->}[u]_{s'} \ar@{->>}[r]_{y} & Y \ar@<-1ex>@{ >->}[u]_{s}  }\]
	is \emph{a regular pushout}, namely such that the factorization from $X'$ to the pullback of $f$ along $y$ is a regular epimorphism.
\end{corollary}

\begin{proof}
Clearly if this condition holds, it holds in particular for horizontal split epimorphisms. Then the conclusion is given by Corollary~\ref{regpush1}. Conversely, suppose $\EE$ is a regular Mal'tsev category, and complete the square by the horizontal kernel equivalence relations:
$$\xymatrix@=9pt{
	\Eq[x] \ar@<-1ex>[ddd]_{\Eq(f')} \ar[rd]_{\chi} \ar[rrrr]_<<<<<<<<{d_{1}^x}\ar@<2ex>[rrrr]^<<<<<<<<{d_{0}^x} &&&&  {X'\;} \ar[rd]^{\psi} \ar@<-1ex>[llll] \ar@<1ex>[llld] \ar@{->>}@<1ex>[rrrr]^x \ar@<-1ex>[ddd]_>>>>>{f'}   &&&& X \ar[ddd]_{f} \\
	&  \bar P \ar[ddl] \ar@<-1ex>[rrrr]_<<<<<<<<{\delta_{1}} \ar[rrru]^<<<<<<<<{\delta_{0}} &&&&  P 
	\ar[ddl]_{} \ar@{->>}[rrru]^<<<<{\bar q}  & & &  &   && \\
	&&& &&&&&&\\
	\Eq[y] \ar[uuu]_(.55){\Eq(s')} \ar@<-1ex>[rrrr]_{d_{1}^y} \ar@<1ex>[rrrr]^{d_{0}^y} \ar@<-1ex>[uur] &&&&  Y' \ar[llll] \ar@{->>}[rrrr]_{y}  \ar[uuu]_<<<<<{s'}
	\ar@<-1ex>[uur]_{} && && Y   \ar@<-1ex>[uuu]_{s}
}
$$
Then denote by $\psi$ (resp. $\chi$) the factorization from $X'$ to the pullback of $f$ along $y$ (resp. from $\Eq[x]$ to the pullback of $f'$ along $d_0^y$). The map $\chi$ is a regular epimorphism according to Corollary~\ref{regpush1}. Moreover the quadrangle $x\cdot \delta_0 = \bar q \cdot \delta_1$ is a pullback, and, $\EE$ being regular, the factorization $\delta_1$ is a regular epimorphism, since so is $x$. Then the equality $\psi\cdot d_1^x=\delta_1\cdot\chi$ shows that $\psi$ is a regular epimorphism, since so is $\delta_1\cdot\chi$.
\end{proof}

As mentioned in the first section, in the case of a variety $\mathbb V$ of universal algebras, the Mal'tsev property can be expressed by a ternary term $p(x,y,z)$ satisfying the identities $p(x,y,y) = x$ and $p(x,x,y)=y$~\cite{Maltsev}. We shall prove the existence of such a term by adopting a categorical approach, first considered in~\cite{CP}, based on an interpretation of a suitable regular pushout lying in the full subcategory of free algebras.

\begin{theorem}{}\label{MaltseVar}
A variety $\mathbb V$ of universal algebras is a Mal'tsev category if and only if its algebraic theory has a ternary term $p$ satisfying the identities $p(x,y,y)= x$ and $p(x,x,y)=y$.
\end{theorem}

\begin{proof}
In Theorem~\ref{CharacterizationM} it is shown that the existence of the Mal'tsev term $p$ implies that the variety is a Mal'tsev category. 

Conversely, assume that $\mathbb V$ is a Mal'tsev variety, and denote by $X$ the free algebra on one element, by $X+X$ the free algebra on two elements, and by $X+X+X$ the free algebra on three elements. If $\nabla \colon X + X \rightarrow X$ is the codiagonal, then
the following diagram commutes:
$$
\xymatrix@=20pt{X + X + X \ar[r]^-{\nabla + 1_X} \ar[d]_-{1_X + \nabla} & X + X \ar [d]^{\nabla} \\
X + X \ar[r]_{\nabla} & X
}
$$
This diagram is clearly a regular pushout by Corollary~\ref{regpush2}, so that the canonical factorization $\alpha \colon X + X + X \rightarrow \Eq[\nabla]$ to the kernel pair of~$\nabla$ in~$\mathbb V$ is a regular epimorphism, i.e. a surjective homomorphism. We can then choose the element $(q,r) \in \Eq[\nabla]$, where $q(x,y)= x$ and $r(x,y)= y$, and we know that there is a ternary term $p(x,y,z) \in X + X + X$ with $\alpha(p) = (q,r)$. 
Consider then the following commutative diagram  $$
\xymatrix@=30pt{X + X & X + X + X \ar[r]^-{\nabla + 1_X}  \ar@{.>}[d]_{\alpha}  \ar[l]_-{1_X + \nabla} & X +X \\
& \Eq[\nabla] \ar[ru]_{p_1}  \ar[lu]^{p_0} & 
}
$$ where $p_0$ and $p_1$ are the projections of the kernel pair $\Eq[\nabla]$. When applied to the term $p$, its commutativity exactly expresses the announced identities $p(x,y,y)= x$ and $p(x,x,y)=y$ for the term $p$.
\end{proof}

\begin{remark}
The categorical notion of \emph{regular pushout}, introduced in full generality in~\cite{Bourn3x3} in relationship with the $3\times 3$ Lemma, is also related to the notion of \emph{double extension}~\cite{Double}, that was first considered by G.\ Janelidze in the category of groups. This latter notion has turned out to play a central role in the theory of (higher) central extensions of an exact Mal'tsev category. Indeed, the possibility of inductively defining higher dimensional categorical Galois structures starting from a Birkhoff reflective subcategory of an exact Mal'tsev category also depends on the existence of double extensions and their higher versions (see~\cite{JK, EGV, Everaert, DG} and the references therein). For instance, the higher homology of groups, compact groups and crossed modules can be better understood from this categorical perspective, and many new computations can be made thanks to the characterizations of the higher central extensions relative to the higher dimensional Galois structures.
\end{remark}

\begin{remark}
The essence of the definition of regular categories is to capture the categorical properties of $\mathsf{Set}$ which concern finite limits and regular epimorphisms. This has been formalized by Barr's embedding theorem~\cite{Barr2} which claims that for any small regular category $\EE$ there exists a fully faithful left exact embedding into a presheaf category $\EE \hookrightarrow \mathsf{Set}^{\mathbb C}$ which preserves the regular epimorphisms. Since in a presheaf category limits and quotients are computed componentwise, with this embedding theorem it is enough to prove some statements about finite limits and regular epimorphisms in $\mathsf{Set}$ (i.e. using elements) in order to prove it in full generality for any regular category, see~\cite{BB} for more details. This embedding theorem has been extended to the regular Mal'tsev case in~\cite{jacqmin1}. An essentially algebraic (i.e. locally presentable) regular Mal'tsev category $\mathbb M$ is constructed such that any small regular Mal'tsev category $\EE$ admits a conservative left exact embedding $\EE \hookrightarrow {\mathbb M}^{\mathbb C}$ which preserves the regular epimorphisms. This category $\mathbb M$ is constructed via some partial operations and `approximate Mal'tsev operations'~\cite{approx}. In the same way as with Barr's embedding theorem, one can now reduce the proof of some statements about finite limits and regular epimorphisms in any regular Mal'tsev category to the particular case of $\mathbb M$ and thus use elements and (approximate) Mal'tsev operations. Similar embedding theorems also hold in the regular unital and strongly unital case, see~\cite{jacqmin2,jacqminthesis}. Using partial Mal'tsev operations, one also has an embedding theorem for (non necessarily regular) Mal'tsev categories~\cite{jacqmin3,jacqminthesis}.
\end{remark}

We now observe that, in any exact Mal'tsev category~$\EE$, the category $\mathsf{Cat}(\EE)= \mathsf{Grpd}(\EE)$ of internal categories (=internal groupoids) and internal functors inherits the exactness property from the base category $\EE$. The category $\mathsf{Cat}^n(\EE)$ of $n$-fold internal categories is defined by induction by $\mathsf{Cat}^1(\EE)=\mathsf{Cat}(\EE)$ and $\mathsf{Cat}^{n+1} (\EE) = \mathsf{Cat}(\mathsf{Cat}^n(\EE))$ for $n\ge1$.
 
\begin{theorem}\cite{Gran}
Let $\EE$ be an exact Mal'tsev category. Then: \begin{enumerate}
\item the category $\mathsf{Cat}(\EE)$ of internal categories in $\EE$ is exact Mal'tsev;
\item the category $\mathsf{Cat}^n(\EE)$ of $n$-fold internal categories in $\EE$ is exact Mal'tsev, for any $n\ge1$.
\end{enumerate}
\end{theorem}

\begin{proof}
\begin{enumerate}
\item As shown in~\cite{CPP} the category $\mathsf{Cat}(\EE)$ is a full subcategory of the category $\mathsf{RG}(\EE)$ of reflexive graphs in $\EE$ (Lemma~\ref{subreflexive}). Next, given any internal functor $(f_0, f_1) \colon \mathbb X \rightarrow \mathbb Y$ in $\mathsf{Cat}(\EE)$
$$ \xymatrix@=17pt{X_1 \times_{X_0} X_1 \ar@<-1,ex>[d]_{p_0} \ar@<+1,ex>[d]^{p_1}  \ar[d]|(.50){m}  \ar@{.>}[rr]^{f_2}& & Y_1 \times_{Y_0} Y_1   \ar@<-1,ex>[d]_{p_0} \ar@<+1,ex>[d]^{p_1}  \ar[d]|(.50){m}  \\
	X_1\ar@<-1,ex>[d]_{d_0}\ar@<+1,ex>[d]^{d_1} \ar[rr]^{f_1}
	&& Y_1 \ar@<-1,ex>[d]_{d_0}\ar@<+1,ex>[d]^{d_1}\\
	X_0  \ar[rr]_{f_0} \ar[u]_{}& & Y_0 \ar[u]_{}   
}
$$
it has a canonical factorization in the category $\mathsf{RG}(\EE)$ of reflexive graphs as 
$$ \xymatrix@=17pt{
	X_1\ar@<-1,ex>[d]_{d_0}\ar@<+1,ex>[d]^{d_1} \ar@{->>} [rr]^{q_1}
	& & {I_1\, } \ar@<-1,ex>[d]_{d_0}\ar@<+1,ex>[d]^{d_1}  \ar@{>->}[rr]^{i_1}& & Y_1 \ar@<-1,ex>[d]_{d_0}\ar@<+1,ex>[d]^{d_1}\\
	X_0  \ar@{->>}[rr]_{q_0} \ar[u]_{}& &{I_0\, } \ar[u]_{}  \ar@{>->}[rr]_{i_0}& & Y_0 \ar[u]_{}   
}
$$
where $f_0 =  i_0 \cdot q_0$ and $f_1 = i_1 \cdot q_1 $ are the (regular epi)-mono factorizations of $f_0$ and of $f_1$ in $\EE$, respectively. The induced reflexive graph $\mathbb I$ in the middle of the diagram above is underlying a groupoid structure (by Lemma~\ref{subreflexive}, for instance), and the factorization above is then the (regular epi)-mono factorization in $\mathsf{Cat}(\EE)$ of the internal functor $(f_0, f_1)$. These factorizations are clearly pullback stable in $\mathsf{Cat}(\EE)$, since regular epimorphisms in $\EE$ are pullback stable by assumption. One then checks that any internal equivalence relation in $\mathsf{Cat}(\EE)$ is a kernel pair (see Theorem~$3.2$ in~\cite{Gran}) to conclude that $\mathsf{Cat}(\EE)$ is an exact category. The fact that $\mathsf{Cat}(\EE)$ is a Mal'tsev category immediately follows from Lemma~\ref{conser} and the fact that the forgetful functor $\mathsf{Cat} (\EE) \rightarrow \mathsf{RG}(\EE)$ to the Mal'tsev category $\mathsf{RG}(\EE)$ is left exact and conservative.
\item By induction, this follows immediately from the first part of the proof.
\end{enumerate}
\end{proof}

This result shows a difference with the case of a general exact category $\EE$, for which the category $\mathsf{Cat}(\EE)$ is not even regular, in general. For instance, the ordinary category $\mathsf{Cat}(\mathsf{Set})= \mathsf{Cat}$ of small categories and functors is not regular (the same can be said for the category of small groupoids).

Let us conclude this section by mentioning the important result in~\cite{CKP} asserting that a regular category $\mathbb C$ is a Mal'tsev category if and only if any simplicial object in $\mathbb C$ is an internal Kan complex.

\section{Regular Mal'tsev categories and the calculus of relations}\label{RegularRel}

The aim of this section is to briefly recall the calculus of relations in a regular category and present some instances of its usefulness in the context of Mal'tsev and Goursat categories~\cite{CLP, CKP}. We shall also give a categorical result concerning the direct product decomposition of an object coming from universal algebra.

Given a relation $\langle r_0,r_1 \rangle \colon {R}\rightarrow {X \times Y}$, its opposite relation $R^{\circ}$ is the relation from $Y$ to $X$ given by the subobject $\langle r_1,r_0 \rangle\colon {R}\rightarrow {Y \times X}$. Given two relations $\langle r_0,r_1 \rangle \colon {R}\rightarrow {X \times Y}$ and $\langle s_0,s_1 \rangle  \colon {S}\rightarrow {Y \times Z}$ in a regular category, their composite $SR={S\circ R}\rightarrow {X \times Z}$ can be defined as follows: take the pullback 
$$\xymatrix{R \times_Y S \ar[r]^-{\pi_1} \ar[d]_{\pi_0}  & S \ar[d]^{s_0}\\
	R \ar[r]_{r_1} & Y
}
$$
of $r_1$ and $s_0$, and the (regular epi)-mono factorization of $\langle r_0\pi_0,s_1\pi_1\rangle$:
$$
\xymatrix@R=5pt{R\times_Y S \ar[rr]^-{\langle r_0\pi_0,s_1\pi_1\rangle} \ar@{>>}[dr] & & X\times Z. \\
& S\circ R \ar@{ >->}[ur]}
$$
The composite $S \circ R$ is defined as the relation $\xymatrix{ {S\circ R \,\, } \ar@{>->}[r] & X \times Z}$ in the diagram above.
Note that the transitivity of a relation $R$ on an object $X$ can be expressed by the inequality $R \circ R\leqslant R$, and the symmetry by the inequality $R^{\circ}\leqslant R$ (or, equivalently, by $R^\circ=R$). 

In the following, given an arrow $f\colon A \rightarrow B$ in $\EE$, we shall identify it to the relation $\langle 1_A, f \rangle \rightarrow A \times B$ representing its graph. 
For any arrow $f \colon A \rightarrow B$ the corresponding relation is difunctional:$$ f \circ f^{\circ} \circ f = f.$$ 
Note also that $$f \circ f^{\circ} = 1_B$$ if and only if $f$ is a regular epimorphism, while $f^{\circ} \circ f = \Eq[f].$
Finally, with the notations we have introduced, any relation $\langle r_0,r_1 \rangle \colon {R}\rightarrow {X \times Y}$ can be written as the composite $R = r_1 \circ r_0^{\circ}$.

\begin{theorem}\label{permutability}\cite{Meisen, CLP}
For a regular category $\EE$, the following conditions are equivalent:
\begin{itemize}
\item[(M1)] for any pair of equivalence relations $R$ and $S$ on any object $X$ in $\EE$, \\ $S \circ R = R \circ S$;
\item[(M3)] any relation $U$ from $X$ to $Y$ is difunctional; 
\item[(M4)] $\EE$ is a Mal'tsev category;
\item[(M5)] any reflexive relation $R$ on any object $X$ in $\EE$ is symmetric;
\item[(M6)] any reflexive relation $R$ on any object $X$ in $\EE$ is transitive.
\end{itemize}
\end{theorem}

\begin{proof}
$(M1) \Rightarrow (M3)$ As observed above, any relation $$
\xymatrix{&U \ar[dl]_{u_0} \ar[dr]^{u_1} & \\
X& & Y
}
$$ can be written as $U = u_1 \circ u_0^\circ$. The assumption implies in particular that the kernel pairs $\Eq[u_0]$ and $\Eq[u_1]$ of the projections commute in the sense of the composition of relations (on the object $U$):
$$(u_1^\circ \circ u_1) \circ  (u_0^\circ \circ u_0)= (u_0^\circ \circ u_0) \circ  (u_1^\circ \circ u_1).$$ Accordingly, by keeping in mind that the relations $u_0$ and $u_1$ are difunctional:
\begin{eqnarray}
U  & =& u_1 \circ u_0^\circ \nonumber  \\
 & =& (u_1 \circ u_1^\circ \circ u_1) \circ (u_0^\circ \circ u_0 \circ u_0^\circ ) \nonumber  \\
& = & u_1 \circ (u_1^\circ \circ u_1) \circ (u_0^\circ \circ u_0) \circ u_0^\circ \nonumber \\
& = & u_1 \circ (u_0^\circ \circ u_0) \circ (u_1^\circ \circ u_1) \circ u_0^\circ \nonumber \\
& = & (u_1 \circ u_0^\circ) \circ (u_0 \circ u_1^\circ) \circ (u_1 \circ u_0^\circ) \nonumber \\
& = & U \circ U^\circ \circ U. \nonumber 
\end{eqnarray}
$(M3) \Rightarrow (M4)$  This appears already as Theorem~\ref{theounit}. Using the calculus of relations, we can proceed as follows. Let $(U,u_0, u_1)$ be a reflexive relation on an object $X$, so that $1_X \le U$. 
 By difunctionality we have: $$U^\circ = 1_X \circ U^\circ \circ 1_X \le U \circ U^\circ \circ U = U,$$ showing that $U$ is symmetric. On the other hand:
$$ U \circ U = U \circ 1_X \circ U \le U \circ U^\circ \circ U = U,$$ and $U$ is transitive.\\
$(M4) \Rightarrow (M5)$ Clear. \\
$(M5) \Rightarrow (M3)$ Let $(U, u_0, u_1)$ be a relation from $X$ to $Y$.
The relation $$u_0^\circ \circ u_0\circ  u_1^\circ \circ u_1$$ is reflexive, since both 
the kernel pairs $u_0^\circ \circ u_0$ and $u_1^\circ \circ u_1$ are reflexive. By assumption the relation $u_0^\circ \circ u_0\circ  u_1^\circ \circ u_1$ is then symmetric:
$$(u_0^\circ \circ u_0\circ  u_1^\circ \circ u_1)^\circ = u_0^\circ \circ u_0\circ  u_1^\circ \circ u_1.$$ This implies that 
$$ u_1^\circ \circ u_1 \circ u_0^\circ \circ u_0 = u_0^\circ \circ u_0\circ  u_1^\circ \circ u_1,$$
and then, by multiplying on the left by $u_1$ and on the right by $u_0^\circ$ we get the equality
$$ u_1 \circ u_1^\circ \circ u_1 \circ u_0^\circ \circ u_0 \circ u_0^\circ= u_1 \circ u_0^\circ \circ u_0\circ  u_1^\circ \circ u_1 \circ u_0^\circ.$$
By difunctionality of $u_1$ and $u_0^\circ$ it follows that 
$$ u_1 \circ u_0^\circ = u_1 \circ u_0^\circ \circ u_0\circ  u_1^\circ \circ u_1 \circ u_0^\circ,$$
and then 
$$u_1 \circ u_0^\circ =  (u_1 \circ u_0^\circ) \circ (u_1 \circ u_0^\circ )^\circ \circ (u_1 \circ u_0^\circ),$$
showing that $U = u_1 \circ u_0^\circ$ is difunctional.

Observe that $(M4) \Rightarrow (M6)$ is obvious, and let us then prove that $(M6) \Rightarrow (M3)$.
Let $U= u_1 \circ u_0^\circ$ be any relation from $X$ to $Y$.
The relation $$u_1^\circ \circ u_1\circ  u_0^\circ \circ u_0$$ is reflexive, thus it is transitive by assumption. This gives the equality
$$(u_1^\circ \circ u_1\circ  u_0^\circ \circ u_0) \circ (u_1^\circ \circ u_1\circ  u_0^\circ \circ u_0) = u_1^\circ \circ u_1 \circ  u_0^\circ \circ u_0, $$ yielding 
$$u_1 \circ u_1^\circ \circ u_1\circ  u_0^\circ \circ u_0 \circ u_1^\circ \circ u_1\circ  u_0^\circ \circ u_0 \circ u_0^\circ= u_1 \circ u_1^\circ \circ u_1\circ  u_0^\circ \circ u_0 \circ u_0^\circ.$$
By difunctionality we conclude that
$$ u_1 \circ  u_0^\circ \circ u_0 \circ u_1^\circ \circ u_1\circ u_0^\circ= u_1 \circ u_0^\circ,$$
and $$U \circ U^\circ \circ U =U.$$

Finally, to see that $(M5)  \Rightarrow (M1)$, observe that the relation $S \circ R$ is reflexive, and then it is symmetric, so that 
$$S \circ R= (S \circ R)^\circ = R^\circ \circ S^\circ = R \circ S,$$
concluding the proof.
\end{proof}

 \subsection*{Direct product decompositions}
 
In any regular category $\mathbb E$, given two equivalence relations $R$ and $S$ on $X$ such that $R\circ S = S \circ R$, the composite $R \circ S$ is then an equivalence relation: indeed, the relation $R \circ S$ is obviously reflexive, but also symmetric, since $$(R\circ S)^\circ = S^\circ \circ R^\circ = S \circ R = R \circ S,$$ and transitive: $$(R\circ S) \circ (R\circ S) = R \circ R \circ S \circ S = R \circ S.$$
The equivalence relation $R \circ S$ is then the supremum $R \vee S$ of $R$ and $S$ as equivalence relations on $X$. 
When this is the case, by Proposition $2.3$ in~\cite{direct-product}, the canonical morphism $\alpha \colon R \square S \rightarrow R \times_XS$ in the following diagram: 
$$
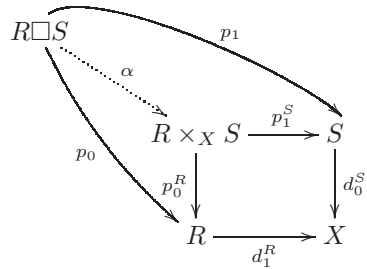
\xymatrix@=25pt{R \square S \ar@{.>}[dr]^{\alpha} \ar@(u,r)[drr]^{p_1} \ar@/_/[ddr]_{p_0}&  & \\
& R \times_X S \ar[r]^-{p_1^S} \ar[d]_-{p_0^R} & S \ar[d]^-{d_0^S} \\ & R \ar[r]_-{d_1^R} & X}$$
from the largest double equivalence relation $R \square S$ on $R$ and $S$ to the pullback $R \times_X S$ is a regular epimorphism.
We then get the following:

\begin{theorem}\label{direct-product}\cite{direct-product}
Let $\mathbb E$ be an exact category, $R$ and $S$ two equivalence relations on $X$ such that:
\begin{itemize}
\item $R \wedge S = \Delta_X$;
\item $R \circ S = S \circ R$ and
\item $R \vee S = \nabla_X$.
\end{itemize}
Then $X$ is isomorphic to $X/R \times X/S$.
\end{theorem}

\begin{proof}
The first two assumptions imply that any of the commutative squares on the left hand side in the diagram
\begin{equation}\label{quotient}
 \xymatrix@=33pt{
	R \square S \ar@<-1,ex>[d]_{p_0}\ar@<+1,ex>[d]^{p_1}  \ar@<1,ex>[r]^-{p_0} \ar@<-1,ex>[r]_-{p_1}
	& S \ar[l]\ar@<-1,ex>[d]_{s_0}\ar@<+1,ex>[d]^{s_1} \ar@{->>}[r]^{\tau} & T \ar@<-1,ex>[d]_{t_0}\ar@<+1,ex>[d]^{t_1}  \\
	R \ar@<-1,ex>[r]_-{r_1} \ar@<1,ex>[r]^-{r_0} \ar[u]& X \ar[u] \ar@{->>}[r]_{q_R} \ar[u] \ar[l] & X/R \ar[u]
}
\end{equation}
is a pullback (since the canonical morphism $R \square S \rightarrow R \times_XS$ in~\eqref{canonical} is both a monomorphism and a regular epimorphism). The right-hand part of the diagram is obtained by taking the quotient $X/R$ of $X$ by the equivalence relation $R$, and the quotient $T$ of $S$ by the equivalence relation $R \square S$ on $S$, with $t_0$ and $t_1$ the induced factorizations. The fact that the equivalence relation $R \square S$ on $S$ is the inverse image of the relation $R \times R$ on $X \times X$ implies that $(t_0, t_1) \colon T \rightarrow X/R \times X/R$ is a monomorphism. The relation $T$ actually is an equivalence relation (by Theorem $3$ in~\cite{BournComprehensive}), and the so-called Barr-Kock theorem~\cite{Barr, BG2} implies that the following square is a pullback
$$
\xymatrix{ X \ar@{->>}[r]^{q_R} \ar@{->>}[d]_{q_S} & X/R \ar@{->>}[d]^{\gamma} \\
X/S \ar@{->>}[r]_{\beta} & Q
}
$$
where $\gamma \colon X/R \rightarrow Q$ is the quotient of $X/R$ by $T$ and $\beta \colon X/S \rightarrow Q$ the unique induced factorization. This square is also a pushout (since $\tau$ in the diagram~\eqref{quotient} is a regular epimorphism), and in the exact category $\mathbb E$ this implies that the kernel pair $\Eq[\gamma \cdot q_R]$ of $\gamma \cdot q_R$ is the supremum $R \vee S$ of $R$ and $S$ as (effective) equivalence relation on $X$. Since $R \vee S = \nabla_X $, we conclude that $Q$ is the quotient of $X$ by $\nabla_X$, therefore it is a subobject of the terminal object $1$.
Accordingly, the following diagram is a pullback
$$
\xymatrix{ X \ar@{->>}[r]^{q_R} \ar@{->>}[d]_{q_S} & X/R \ar[d]^{} \\
X/S \ar[r]_{} & 1
}
$$
and $X \cong X/R \times X/S$, as expected. 
\end{proof}

\begin{remark}
In any exact category $\mathbb E$, the product $\xymatrix{X & X \times Y \ar[r]^-{p_1} \ar[l]_-{p_0}& Y}$
of two objects $X$ and $Y$ is such that 
$\Eq[p_0] \wedge \Eq[p_1] = \Delta_{X \times Y}$, $\Eq[p_0] \circ \Eq[p_1] = \Eq[p_1] \circ \Eq[p_0]$ and $\Eq[p_0] \vee \Eq[p_1] = \nabla_{X \times Y}$. Theorem~\ref{direct-product} can then be seen as a kind of converse to this simple observation.
\end{remark}

\begin{remark}
We observe that the assumptions in Theorem~\ref{direct-product} are the (categorical formulations of the) properties defining a pair of \emph{factor congruences} in the sense of universal algebra.
\end{remark}

If the base category $\mathbb E$ is exact Mal'tsev we immediately get the following:

\begin{corollary}
Let $\mathbb E$ be an exact Mal'tsev category. Whenever two equivalence relations $R$ and $S$ on the same object $X$ are such that $R \wedge S = \Delta_X$ and $R \vee S = \nabla_X$, then there is a canonical isomorphism $X \cong X/R \times X/S$.
\end{corollary}

\subsection*{A glance at Goursat categories}

We now briefly recall and study some basic properties of Goursat categories. The origin of this important concept definitely goes back to the celebrated article~\cite{CLP} by A. Carboni, J. Lambek and M.C. Pedicchio, although the explicit definition and a first systematic study of Goursat categories was presented later in~\cite{CKP}.

\begin{definition}
A regular category $\mathbb E$ is a Goursat category if for any two equivalence relations $R$ and $S$ on the same object $X$ in $\mathbb E$ one has the equality $$R \circ S \circ R = S \circ R \circ S.$$
\end{definition}

Remark that a variety of universal algebras is a Goursat category if and only if it is $3$-permutable in the usual sense~\cite{HH}. Any regular Mal'tsev category is clearly a Goursat category, however the converse is not true: indeed, the variety of implication algebras is an example of a $3$-permutable variety, and therefore of an exact Goursat category, that is not $2$-permutable~\cite{Mitschke}.

A remarkable categorical property of Goursat categories is that the regular image of any equivalence relation is again an equivalence relation. This is actually characteristic of these categories, as shown in~\cite{CKP}. Here below we give a simple and self-contained proof of this result:

\begin{proposition}\label{regular-image}
A regular category $\mathbb E$ is a Goursat category if and only if for any equivalence relation $R$ on any object $X$ and any regular epimorphism $f \colon X \twoheadrightarrow Y$, the regular image $f(R)$ is an equivalence relation.
\end{proposition}

\begin{proof}
One implication is direct: in a regular category $\mathbb E$ the relation $f(R)$ is always reflexive and symmetric, and when $\mathbb E$ is a Goursat category then it is also transitive: \begin{eqnarray} &f(R) \circ  f(R) = (f \circ R \circ f^\circ) \circ (f \circ R \circ f^\circ) \nonumber \\
& = f \circ f^\circ \circ f \circ R \circ f^\circ \circ f \circ f^\circ \nonumber = f \circ R \circ f^\circ = f (R). \nonumber
\end{eqnarray}
Note that the assumption has been used in the second equality.

For the converse, consider two equivalence relations $R$ and $S$ on $X$. Then the composite $R\circ S \circ R$ can be written as 
$$R \circ S \circ R = r_1 \circ r_0^\circ \circ s_1 \circ s_0^\circ \circ r_1 \circ r_0^\circ= r_1 \circ r_0^\circ \circ s_1 \circ s_0^\circ \circ r_0 \circ r_1^\circ = r_1 (r_0^{-1} (S)),$$
as observed in~\cite{BournGoursat}. The assumption then implies that $R \circ S \circ R$ is an equivalence relation, as a direct image of the equivalence relation $r_0^{-1} (S)$ along the split epimorphism $r_1$ (which is then a regular epimorphism). Its transitivity implies that $S \circ R \circ S \le R \circ S \circ R$, and then $S \circ R \circ S = R \circ S \circ R$.
\end{proof}

This characterization and the so-called \emph{denormalized $3 \times 3$-Lemma}~\cite{Bourn3x3, Lack} inspired a new characterization of Goursat categories in terms of a special kind of pushouts:

\begin{definition}\label{GoursatPush}\cite{Gran-Rodelo}
Consider a commutative square (with $y \cdot f = f' \cdot x$ and $x \cdot s = s' \cdot y$)
\begin{equation}\label{GPushout}
\xymatrix@C=2,5pc@R=2pc{ 
		X \ar@<-1ex>[d]_{f} \ar@{->>}[r]^{x} 
		& X' \ar[d]_{f'} \\
		Y \ar[u]_{s} \ar@{->>}[r]_{y} & Y' \ar@<-1ex>[u]_{s'}  }
		\end{equation}
where the vertical morphisms are split epimorphisms and the horizontal ones are regular epimorphisms. This square is a pushout, and it is called a \emph{Goursat pushout} if the induced morphism $\Eq[f] \rightarrow \Eq[f']$ is a regular epimorphism.
\end{definition}

\begin{proposition}\label{Gran-Rodelo}\cite{Gran-Rodelo}
A regular category $\mathbb E$ is a Goursat category if and only if any commutative diagram~\eqref{GPushout} is a Goursat pushout.
\end{proposition}

\begin{proof}
If $\mathbb E$ is a Goursat category we know that the regular image $x(\Eq[f])$ can be computed as follows:
\begin{align*} &x(\Eq [f]) = x  \circ f^\circ \circ f \circ x^\circ = x  \circ x^\circ \circ x  \circ f^\circ \circ f \circ x^\circ \circ x \circ x^\circ\\
&= x  \circ  f^\circ \circ f \circ  x^\circ \circ x  \circ  f^\circ \circ f \circ x^\circ = x  \circ  f^\circ \circ y^\circ \circ y \circ f  \circ x^\circ\\
&= x  \circ  x^\circ \circ {f'}^\circ  \circ f' \circ x \circ x^\circ = {f'}^\circ  \circ f' = \Eq[f']
\end{align*}
where we have used the Goursat assumption, the commutativity of the diagram~\eqref{GPushout}, and fact that $x$ is a regular epimorphism, so that $x  \circ  x^\circ = 1_{X'}$.

For the converse, consider an equivalence relation $S$ on an object $X$ and a regular epimorphism $f \colon X \twoheadrightarrow Y$. The regular image $f(S) = T$ is certainly reflexive and symmetric, and by Proposition~\ref{regular-image} it suffices to show that it is also transitive.
Since $S$ is symmetric and transitive, we know that there exists a morphism $\tau_S$ such that the diagram $$ \xymatrix@=25pt{
	\Eq[s_0] \ar@<-1,ex>[d]_{p_0}\ar@<+1,ex>[d]^{p_1} \ar@{.>} [r]^-{\tau_S}
	& S \ar@<-1,ex>[d]_{s_0}\ar@<+1,ex>[d]^{s_1}\\
	S \ar[r]_{s_1}\ar[u] & X \ar[u]
}
$$
commutes. Moreover, the diagram 
$$ \xymatrix@C=2,5pc@R=2pc{ 
		S \ar@<-1ex>[d]_{s_0} \ar@{->>}[r]^-{\overline{f}} 
		& f(S) = T \ar[d]_{t_0} \\
		X \ar[u]_{} \ar@{->>}[r]_{f} & Y\ar@<-1ex>[u]_{}  }$$
		is of type~\eqref{GPushout}, the upward pointing arrows being the morphisms giving the reflexivity of $S$ and $T$, respectively. It follows that the factorization $\tilde{f}$ induced by the universal property of the kernel pair $\Eq[t_0]$ - and making the square
		$$ \xymatrix{ 
		\Eq[s_0]  \ar[d]_{\overline{f} \tau_S} \ar@{->>}[r]^{\tilde{f}} 
		& \Eq[t_0] \ar@{.>}[ld]_{\tau_T} \ar[d]^{(t_1 \times t_1)(p_0,p_1)} \\
		{T \, \, } \ar@{>->}[r]_-{(t_0, t_1)} & Y \times Y}$$
commute - is a regular epimorphism by the assumption. Since $(t_0, t_1) $ is a monomorphism, it follows that there is a unique morphism $\tau_T$ making the diagram above commute, and $\tau_T$ makes the diagram
$$ \xymatrix@=25pt{
	\Eq[t_0]  \ar@<-1,ex>[d]_{p_0}\ar@<+1,ex>[d]^{p_1} \ar@{.>} [r]^-{\tau_T}
	& T \ar@<-1,ex>[d]_{t_0}\ar@<+1,ex>[d]^{t_1}\\
	T  \ar[r]_{t_1}\ar[u] & Y\ar[u]
}
$$
commute. It follows that the relation $T$ is transitive, as desired.
 \end{proof}
 
 When $\mathbb V$ is a variety of universal algebras, a direct application of this characterization to a suitable Goursat pushout in the category of free algebras - a similar argument to the one used above to prove Theorem~\ref{MaltseVar} - yields a categorical proof (see~\cite{Gran-Rodelo}) of the following well known Theorem.

\begin{theorem}
 For a variety $\mathbb V$ the following conditions are equivalent:
 \begin{enumerate}
 \item $\mathbb V$ is a $3$-permutable variety;
 \item in the algebraic theory of $\mathbb V$ there are two quaternary terms $p$ and $q$ satisfying the identities $p(x,y,y,z)=x$, 
 $q(x,y,y,z)=z$, and $p(x,x,y,y) = q(x,x,y,y)$.
 \end{enumerate}
 \end{theorem}
 
 More generally, for any $n \ge2$, other characterizations of $n$-per\-mut\-able varieties in terms of ternary operations and identities are considered in~\cite{GR2, JR} using categorical arguments (see also the references therein).
 
 We conclude this section with another characterization of Goursat categories, whose proof is based on the calculus of relations and on the notion of Goursat pushout. It concerns commutative diagrams of the form 
  \begin{equation}\label{3x3} \xymatrix@=25pt{
	\Eq[\phi] \ar@<-1,ex>[d]_{p_0}\ar@<+1,ex>[d]^{p_1} \ar@<-1,ex>[r]_{\overline{h}_0}\ar@<+1,ex>[r]^{\overline{h}_1}
	& \Eq[f]  \ar@<-1,ex>[d]_{p_0}\ar@<+1,ex>[d]^{p_1} \ar[r]^{\overline{h}}\ar[l] & \Eq[g] \ar@<-1,ex>[d]_{p_0}\ar@<+1,ex>[d]^{p_1} \\
	\Eq[h] \ar@{->>}[d]_{\phi} \ar@<-1,ex>[r]_{p_0} \ar@<+1,ex>[r]^{p_1} \ar[u] & A \ar[u]\ar@{->>}[d]_{f} \ar@{->>}[r]^h \ar[l]& C\ar[u] \ar@{->>}[d]^{g} \\
	K  \ar@<-1,ex>[r]_{k_0} \ar@<+1,ex>[r]^{k_1} &  B \ar@{->>}[r]_k \ar[l] & D
}
\end{equation}
(i.e. for any $i, j \in \{0,1\}$, $p_i p_j = p_j \overline{h}_i$, $h p_j =  p_j \overline{h}$, $f p_i = k_i \phi$, and $kf = gh$) where the three columns and the middle row are exact (i.e. regular epimorphisms equipped with their kernel pairs): 

\begin{theorem}\label{Lack, Gran-Rodelo}\cite{Lack, Gran-Rodelo}
 For a regular category $\mathbb E$ the following conditions are equivalent:
 \begin{itemize}
 \item $\mathbb E$ is a Goursat category;
 \item the \emph{Upper $3 \times 3$ Lemma} holds in $\mathbb E$: given any commutative diagram~\eqref{3x3},
the upper row is exact whenever the lower row is exact;
\item the \emph{Lower $3 \times 3$ Lemma} holds in $\mathbb E$: given any commutative diagram~\eqref{3x3}, 
the lower row is exact whenever the upper row is exact;
\item the \emph{$3 \times 3$ Lemma} holds in $\mathbb E$: given any commutative diagram~\eqref{3x3}, 
the lower row is exact if and only if the upper row is exact.
\end{itemize}
\end{theorem}
 
 Note that this homological lemma was not foreseen in the original project in~\cite{CLP}.
 Regular Mal'tsev categories can also be characterized by a stronger version of the denormalized $3 \times 3$ Lemma, called the \emph{Cuboid Lemma}~\cite{Cuboid}.

\begin{remark}
In a similar way as Mal'tsev categories were first defined in the regular context and later studied in the finitely complete context, Goursat categories can be defined without the assumption of regularity, see~\cite{Bourn14}.
\end{remark}

\section{Baer sums in Mal'tsev categories}\label{section Baer sums in Mal'tsev categories}

In this last section, we shall be interested in those extensions, namely regular epimorphisms $f\colon X\onto Y$, which have abelian kernel equivalence relations, and we shall show that, from the Mal'tsev context, emerges a very natural notion of Baer sums. Such an extension is actually nothing but an affine object with global support in the slice category  $\EE/Y$. So, we shall show that, in any Mal'tsev category being sufficiently exact, we are able to associate, with any affine object with global support, an abelian object, called its \emph{direction}, and to show  as well that the set (up to isomorphisms) of the affine objects with global support and a given direction $A$ is endowed with a canonical abelian group structure on the (non-Mal'tsev) general model of~\cite{Bourn3}.
By sufficiently exact, we mean the following:

\begin{definition}\cite{Bourn5}
	A regular category $\EE$ is said to be \emph{efficiently regular} when any equivalence relation $R$ on an object $X$ which is a subobject $i\colon R\into \Eq[f]$ of an effective equivalence relation is effective as well as soon as the monomorphism $i$ is regular, i.e. is the equalizer of some pair of morphisms. 
\end{definition}

The categories $\mathsf{Ab}(\mathsf{Top})$ and $\mathsf{Gp}(\mathsf{Top})$ of (resp. abelian) topological groups are examples of non-exact efficiently regular categories. The major interest of such a category is that any discrete fibration between equivalence relations $R\to \Eq[g]$ makes $R$ effective as well~\cite{Bourn5}. Note that this latter property could also be guaranteed by the assumption that the base category $\EE$ is regular and that regular epimorphisms in $\EE$ are effective descent morphisms (see~\cite{JST}, for instance, for more details).

\emph{In this section we shall suppose that $\EE$ is an efficiently regular Mal'tsev category}. Take now any affine object $X$ with global support (namely such that the terminal map $\tau_X\colon  X\to 1$ is a regular epimorphism) and consider the following diagram where $p\colon  X\times X\times X\to X$ is the internal Mal'tsev operation on $X$ giving rise to the affine structure:

$$ \xymatrix@=25pt{
	X\times X \times X \ar@<-1,ex>[dd]_{p_0^X}\ar@<+1,ex>[dd]^{(p,p_1^X\cdot p_1^X)} \ar@<-1,ex>[r]_-{p_2^X}\ar@<+1,ex>[r]^-{(p_0^X\cdot p_1^X,p)}
	& X\times X \ar@<-1,ex>[dd]_{p_0^X}\ar@<+1,ex>[dd]^{p_1^X} \ar[l] \ar@{.>>}[r]^-{q_p^X} &  A \ar@{.>>}@<+1,ex>[dd]^{\tau_A}\\
	&&&&\\
	X\times X \ar@<-1,ex>[r]_-{p_1^X} \ar@<+1,ex>[r]^-{p_0^X} \ar[uu]_{} & X\ar@{->>}[r]_-{\tau_X}\ar[l]\ar[uu]_{} & 1 \ar@{.>}[uu]^{0_A} 
}
$$
\begin{definition}\cite{Bourn3}
	The upper horizontal reflexive relation (which is an equivalence relation) is called the Chasles relation $Ch_p$ associated with the internal Mal'tsev operation $p$.
\end{definition}

In set-theoretic terms we get $(x,p(x,y,z))Ch_p(y,z)$ or, in other words, $(x,y)Ch_p(x',y')$ if and only if $y=p(x,x',y')$. The operation $p$ is commutative (which is the case in any Mal'tsev category) if and only if we get the equivalence:
$$(x,y)Ch_p(x',y') \iff (y',y)Ch_p(x',x)$$

Since $\EE$ is efficiently regular and since the left hand side square indexed by $0$ is a discrete fibration between equivalence relations, the equivalence relation $Ch_p$ is effective, and so has a quotient $q_p^X$ which, since $\tau_X$ is a regular epimorphism, produces the split epimorphism $(\tau_A,0_A)$ and makes the right hand side square a pullback.  The vertical right hand side part is necessarily a group in $\EE$ as a quotient of the vertical groupoid (actually the equivalence relation) $\nabla_X$. This group is abelian by Corollary~\ref{aff}.

\begin{definition}\cite{Bourn3}
	The abelian group $A$ is called the \emph{direction} of the affine object $X$ with global support and will be denoted by $d(X)$.
\end{definition}

\begin{proposition}
	Given any abelian group $A$, its direction is $A$.
\end{proposition}

\begin{proof}
The diagram
$$ \xymatrix@=20pt{
	A\times A \ar@<-1,ex>[d]_{p_0^A}\ar@<+1,ex>[d]^{p_1^A}  \ar[rr]^-{p_1^A-p_0^A} &&  A \ar@<+1,ex>[d]^{\tau_A}\\
	A\ar@{->>}[rr]_{\tau_A}\ar[u]_{}  && 1 \ar[u]^{0_A} 
}
$$
shows that the direction is indeed $A$.
\end{proof}

Since the inclusion $\mathsf{Aff}(\EE) \into \EE$ is full, any morphism $f\colon  X\to X'$ between affine objects produces a morphism $d(f)$ making the following diagram commute:
$$ \xymatrix@=20pt{
	X \times X \ar[d]_-{f \times f} \ar@{->>}[r]^-{q_p^X} & A \ar[d]^-{d(f)}\\
	X' \times X'\ar@{->>}[r]_-{q_p^{X'}}  & A' 
}
$$

\begin{proposition}
	The group homomorphism $d(f)$ is an isomorphism if and only if $f$ is an isomorphism.
\end{proposition}

\begin{proof}
It is clear that if $f$ is an isomorphism, so is $d(f)$. Conversely suppose $d(f)$ is an isomorphism. Consider the following diagram:
$$\xymatrix@=8pt{
	{X\times  X\;} \ar[rrd]_{f\times f}  \ar@<-1ex>[ddd]_{p_0^X} \ar@<1ex>[ddd]^{p_1^X}  \ar[rrrrr]^{q_p^{X}} &&&&&  { A\;} \ar[rrd]^{d(f)}_{\cong}  \ar@<1ex>[ddd]^{}\\
	&& X'\times X' \ar@<-1ex>[ddd]_<<<<{p_0^{X'}} \ar@<1ex>[ddd]^<<<<{p_1^{X'}} \ar[rrrrr]_<<<<<<<<<{q_p^{X'}}  &&&&& A' \ar@<1ex>[ddd]^{}  && \\
	&&& &&&&&&\\ 
	X \ar[rrd]_{f}  \ar@{->>}[rrrrr]^>>>>>{}  \ar[uuu]_{}  &&&&& 1 \ar[uuu]_{} \ar@{=}[rrd]  \\
	&&  X'  \ar@{->>}[rrrrr]_{}  \ar[uuu]_{}  &&&&& 1 \ar[uuu]_{} 
}
$$
The front and back squares indexed by $0$ being pullbacks, and $d(f)$ being an isomorphism, the left hand side quadrangle is a pullback. But $X$ and $X'$ having global supports, the Barr-Kock Theorem makes the following square a pullback
$$ \xymatrix@=30pt{
	X \ar[d] \ar[r]^{f} & X' \ar[d]\\
	1\ar@{=}[r]  & 1 
}
$$
and consequently makes $f$ an isomorphism.
\end{proof}

Moreover the fact that the right hand side square defining $A$ is a pullback gives $X$ the structure of an $A$-torsor \emph{which is controlled by the choice of the quotient map} $q_p^X$. Accordingly we get an internal regular epic discrete fibration $q^X\colon  \nabla_X\onto A$. Let us denote by $\mathsf{AbTors}(\EE)$ the category whose objects are the pairs $(X,q^X)$ where $q^X\colon  \nabla_X\onto A$ is a regular epic discrete fibration above an abelian group $A$ (which obviously implies that the object $X$ is affine with global support and direction $A$), and whose morphism are the pairs $(f\colon  X\to Y,g\colon  A\to B)$ of a morphism $f$ and a group homomorphism $g$ such that $g\cdot q^X=q^Y\cdot (f\times f)$. 

Let us denote by $\mathsf{Aff}_*(\EE)$ the full subcategory of $\mathsf{Aff} (\EE)$ whose objects have a global support. The previous construction produces a functor $\Phi\colon  \mathsf{Aff}_*(\EE) \to \mathsf{AbTors}(\EE)$ which is an equivalence of categories. Furthermore there is an obvious forgetful functor $U\colon \mathsf{AbTors}(\EE)\to \mathsf{Ab}(\EE)$ such that $d=U\cdot \Phi$. We shall now investigate the properties of the functor $U$. We immediately get:

\begin{proposition}
	Given  any efficiently regular Mal'tsev category $\EE$, the functor $U$ is conservative, it preserves the finite products and the regular epimorphisms. It preserves the pullbacks when they exist, and consequently reflects them.	
\end{proposition}

The restriction about the existence of pullbacks comes from the fact that the objects with global support are not stable under pullback in general. Here comes the main result of this section:

\begin{theorem}\cite{Bourn3}
	Given  any efficiently regular Mal'tsev category $\EE$, the functor $U$ is a cofibration. Being also conservative, any map in $\mathsf{AbTors}(\EE)$ is cocartesian, and any fibre $U_A$ above an abelian group $A$ is a groupoid.
\end{theorem}

\begin{proof}
First we shall show that there are cocartesian maps above regular epimorphisms.
Let us start with an object $(X,q^X)$ in $\mathsf{AbTors}(\EE)$ above the abelian group $A$ and a regular epic group homomorphism $g\colon A\onto B$. Let us denote by $R_g\into X\times X$ the subobject $(g\cdot q^X)^{-1}(0)$. It produces an equivalence relation on $X$ which is effective since the monomorphism in question is regular. Let us denote by $q_g\colon X\onto Y$ its quotient. According to Proposition~\ref{regstab}, the object $Y$ is an affine object, with global support, since so is~$X$. 

Let us show now that there is a (necessarily unique)
map $q^Y\colon Y\times Y\to B$ such that $q^Y\cdot (q_g\times q_g)=g\cdot q^X$. Since $q_g\times q_g$ is a regular epimorphism, it is enough to show that $g\cdot q^X$ coequalizes its kernel equivalence relation.  For sake of simplicity we shall set $q^X(x,x')=\overrightarrow{xx'}$. So we have to show that, when we have $xR_gt$ and $x'R_gt'$ (namely $g(\overrightarrow{xt})=0$ and $g(\overrightarrow{x't'})=0$), we get $g(\overrightarrow{xx'})=g(\overrightarrow{tt'})$, which is straightforward. It remains to show that the following square is a pullback of split epimorphisms:
$$ \xymatrix@=20pt{
	Y\times Y \ar@<-1,ex>[d]_{p_0^Y}  \ar@{->>}[r]^-{q^Y} &  B \ar@<+1,ex>[d]^{\tau_B}\\
	Y\ar@{->>}[r]_{\tau_Y}\ar[u]_{s_0^Y} & 1 \ar[u]^{0_B} 
}
$$
First, we can check that the upward square commutes by composition with the regular epimorphism $q_g$. Let us denote by $\psi\colon Y\times Y\to P$ the factorization through the pullback. Since $\EE$ is a Mal'tsev category, $\psi$ is necessarily a regular epimorphism by Corollary~\ref{regpush2}. We can check it is a  monomorphism as well in the following way: consider the kernel equivalence relation $\Eq[\psi\cdot(q_g\times q_g)]$; it is easy to check that it is coequalized  by $q_g\times q_g$. Accordingly $\Eq[\psi]$ is the discrete equivalence relation and $\psi$ is a monomorphism.

Now, given any pair $(X,B)$ of an affine object $X$ with global support and an abelian group $B$, the map $(1_X,0_B)\colon X\to X\times B$ has direction  $(1_{d(X)},0_B)\colon d(X)\to d(X)\times B$, and  it is easy to check that it is cocartesian. Then, starting with any group homomorphism $h\colon d(X)\to B$, we get the following commutative diagram:
$$ \xymatrix@=20pt{
	&  d(X)\times B \ar@<1ex>[rd]^-{<h,1_B>}\\
	d(X)\ar[rr]_h \ar@<1ex>[ur]^-{(1_{d(X)}, 0_B)}  && B \ar@{.>}[lu]^-{(0_{d(X)},1_B)}
}
$$
where the map $<h,1_B>$ comes from the fact that the product is the direct sum as well in the additive category $\mathsf{Ab}(\EE)$. Moreover, this map, being split, is a regular epimorphism. Accordingly the map $h$ has a cocartesian map above it as well.
\end{proof}

\begin{corollary}\cite{Bourn3} 
	Given any efficiently regular Mal'tsev category $\EE$, the fibre $U_A$ above the abelian group $A$ is endowed with a canonical symmetric closed monoidal structure $\otimes_A$ whose unit is $A$.
\end{corollary} 

\begin{proof}
We recalled that $U_A$ is necessarily a groupoid. Given any pair $(X,X')$ of affine objects with global support and direction $A$, the tensor product ${X\otimes_A X'}$ is defined as the codomain of the (regular epic) cocartesian map above $+\colon A\times A \to A$ with domain $X\times X'$. The commutative diagram in $\mathsf{Ab}(\EE)$ expressing the associativity of the group law: $a+(b+c)=(a+b)+c$ produces the desired associative isomorphism $a_{(X,X',X'')}\colon {X\otimes_A(X'\otimes_AX'')}\cong {(X\otimes_AX') \otimes_AX''}$, while the commutative diagram expressing the commutativity of the group law $a+b=b+a$ and the twisting isomorphism $\tau_{(X,X')}\colon X\times X'\cong X'\times X$ produce the symmetric isomorphism $\sigma_{(X,X')}\colon {X\otimes_AX'}\cong X'\otimes_AX$. The unit of this tensor product is determined by the codomain of the cocartesian map with domain $1$ above  $0_A\colon 1\into A$, namely $A$ itself. The left unit isomorphism $A\otimes_AX\cong X$ is produced by the commutative diagram in $\mathsf{Ab}(\EE)$ associated with the left unit axiom $0+a=a$, a similar construction producing the right unit isomorphism.

This monoidal structure is closed since in the abelian context the division map $d(a,b)=b-a$ is a  group homomorphism. We defined $[X,Y]$ as the codomain of the cocartesian map above $d$ with domain $X\times Y$. The commutative diagram determined by $a +(b-a)=b$ induces the isomorphism $X\otimes_A[X,Y]\cong Y$, while the one determined by $(b+a)-b=a$ induces the isomorphism $X\cong [Y,Y\otimes_AX]$.
\end{proof}

Accordingly this produces an abelian group structure on the set $Ext(A)$ of the connected components of the groupoid $U_A$ whose operation is called in classical terms the \emph{Baer sum}. 

Starting with the variety $\mathsf{Mal}$ of Mal'tsev algebras, the subvariety $\mathsf{Aff}(\mathsf{Mal})$ is the variety of associative and commutative Mal'tsev algebras by Proposition~\ref{proposition p autonomous}. An algebra $X$ in $\mathsf{Mal}$ has a global support if and only if it is non-empty. Accordingly the choice of a point in any non-empty affine object makes the fibre $U_A$ a connected groupoid, reduces the group $Ext(A)$ to only one object and makes it invisible.

Now take the example where the Mal'tsev category $\EE$ is the slice category $\mathsf{Gp}/Q$ of the groups above the group $Q$, whose affine objects with global support are the exact sequences with $A$ abelian:
$$ 1\to  A\into X \onto  Q \to 1 $$ 
The direction of this affine object is nothing but the semi-direct product exact sequence produced by the $Q$-module structure on $A$ determined by this exact sequence. And the Baer sum, described above, coincides with the classical Baer sum associated with a given $Q$-module structure on $A$, see for instance Chapter 4, \emph{Cohomology of groups}, in Mac Lane's \emph{Homology}~\cite{MLane}.

To finish this section, let us be a bit more explicit about the construction of $X\otimes_AX'$. In set-theoretic terms, given any pair $(X,X')$ of affine objects with global support, the equivalence relation $R_+$ on $X\times X'$ producing the tensor product is given by $(x,x')R_+(t,t')$ if and only if $\overrightarrow{xt}+\overrightarrow{x't'}=0$, while the direction on $X\otimes_AX'$ is such that $\overrightarrow{\overline{(x,x')},\overline{(z,z')}}=\overrightarrow{xz}+\overrightarrow{x'z'}$. 

The inverse of an affine object $X$ with global support and direction $A$ is the affine object $X^*=[X,A]$, namely the
quotient of $X\times A$ by the equivalence relation $R_d$ defined by $(x,a)R_d(x',a')$ if and only if $a'-a-\overrightarrow{xx'}=0$. The direction of $X^*$ is such that $\overrightarrow{\overline{(x,a)},\overline{(z,b)}}=b-a-\overrightarrow{xz}$. As expected, we can check that we have an isomorphism $\gamma\colon  X\to X^*$ of affine objects defined by $\gamma(x)=\overline{(x,0)}$, whose direction satisfies: $d(\gamma)(\overrightarrow{xx'})=\overrightarrow{\overline{(x,0)},\overline{(x',0)}}=-\overrightarrow{xx'}$.



\begin{thebibliography}{99}
	
	\bibitem{Barr} 
	\textsc{M. Barr}, Exact categories, \textit{Lecture Notes in Mathematics} \textbf{236} (1971), 1--120.
	
	\bibitem{Barr2}
\textsc{M. Barr}, Representation of categories, \textit{Journal of Pure and Applied Algebra} \textbf{41} (1986), 113--137.
	
	\bibitem{BergerB} \textsc{C. Berger and D. Bourn}, Central reflections and nilpotency in exact Mal'tsev categories, \textit{Journal of Homotopy and Related Structures} \textbf{12} (2017), 765--835.
	
	\bibitem{BB} \textsc{F. Borceux and D. Bourn}, Mal'cev, Protomodular, Homological and Semi-Abelian Categories, \textit{Kluwer, Mathematics and Its Applications} \textbf{566} (2004), 479 pp.

	\bibitem{BournComprehensive} \textsc{D. Bourn}, The shift functor and the comprehensive factorization for internal groupoids, \textit{Cahiers de Topologie et G\'eom\'etrie Diff\'erentielle Cat\'egoriques} \textbf{28} (1987), 197--226.
	
	\bibitem{Bourn6} \textsc{D. Bourn}, Normalization equivalence, kernel equivalence and affine categories, \textit{Lecture Notes in Mathematics} \textbf{1488} (1991), 43--62.
	
	\bibitem{Bourn1} \textsc{D. Bourn}, Mal'cev Categories and fibration of pointed objects, \textit{Applied Categorical Structures} \textbf{4} (1996), 302--327.
	
	\bibitem{Bourn3} \textsc{D. Bourn}, Baer sums and fibered aspects of Mal'cev operations, \textit{Cahiers de Topologie et G\'eom\'etrie Diff\'erentielle Cat\'egoriques} \textbf{40 (4)} (1999), 297--316.
	
	\bibitem{Bourn4} \textsc{D. Bourn}, Intrinsic centrality and associated classifying properties, \textit{Journal of Algebra} \textbf{256} (2002), 126--145.
	
	\bibitem{Bourn3x3} \textsc{D. Bourn}, The denormalized $3 \times 3$ lemma, \textit{Journal of Pure and Applied Algebra} \textbf{177} (2003), 113--129.
		
	\bibitem{Bourn12}  \textsc{D. Bourn}, Commutator theory  in regular Mal'cev categories, in: \textit{Galois theory, Hopf algebras, and Semiabelian categories}, G.~Janelidze, B.~Pareigis and W.~Tholen editors, Fields Institute Communications \textbf{43}, Amer. Math. Soc. (2004), 61--75.
			
	\bibitem{BournGoursat} \textsc{D. Bourn}, Congruence distributivity in Goursat and Mal'cev categories, \textit{Applied Categorical Structures} \textbf{13 (2)} (2005), 101--111.
	
	\bibitem{Bourn5} \textsc{D. Bourn}, Baer sums in homological categories, \textit{Journal of Algebra} \textbf{308} (2007), 414--443.
	
	\bibitem{Bourn11} \textsc{D. Bourn}, On the monad of internal groupoids, \textit{Theory and Applications of Categories} \textbf{28 (5)} (2013), 150--165.

	 \bibitem{Bourn14} \textsc{D. Bourn}, Suprema of equivalence relations and non-regular Goursat categories, \textit{Cahiers de Topologie et G\'eom\'etrie Diff\'erentielle Cat\'e\-go\-ri\-ques}, \textbf{59 (2)} (2018), 142--194.
		
	\bibitem{BG}
	\textsc{D. Bourn and M. Gran}, Centrality and normality in protomodular categories, \textit{Theory and Applications of Categories} \textbf{9 (8)} (2002), 151--165.
	
	\bibitem{BG3}
	\textsc{D. Bourn and M. Gran}, Centrality and connectors in Mal'tsev categories, \textit{Algebra Universalis} \textbf{48} (2002), 309--331.
	
	\bibitem{BG2}
	\textsc{D. Bourn and M. Gran}, \textit{Regular, Protomodular and Abelian Categories}, in Categorical Foundations, edited by M.C. Pedicchio and W.~Tholen, Cambridge University Press (2004).
	
	\bibitem{direct-product} \textsc{D. Bourn and M. Gran}, Normal sections and direct product decompositions, \textit{Communications in Algebra} \textbf{32} (2004), 3825--3842.
	
	\bibitem{approx} \textsc{D. Bourn and Z. Janelidze}, Approximate Mal'tsev operations,  \textit{Theory and Applications of Categories} \textbf{21 (8)}  (2008), 152--171.
	
	\bibitem{CKP} \textsc{A. Carboni, G.M. Kelly and M.C. Pedicchio}, Some remarks on Maltsev and Goursat categories, \textit{Applied Categorical Structures} \textbf{1} (1993), 385--421.
	
	\bibitem{CLP}
	\textsc{A. Carboni, J. Lambek and M.C. Pedicchio}, Diagram chasing in Mal'cev categories, \textit{Journal of Pure and Applied Algebra} \textbf{69} (1990), 271--284.
		
	\bibitem{CP}
	\textsc{A. Carboni and M.C. Pedicchio}, A new proof of the Mal'cev theorem, Categorical studies in Italy (Perugia, 1977), \textit{Rend. Circ. Mat. Palermo} \textbf{2} Suppl. No. 64 (2000), 13--16.
	
	\bibitem{CPP}
	\textsc{A. Carboni, M.C. Pedicchio and N. Pirovano}, Internal graphs and internal groupoids in Mal'tsev categories, \textit{Canadian Mathematical Society Conference Proceedings} \textbf{13} (1992), 97--109.
	
	\bibitem{DG}
	\textsc{A. Duvieusart and M. Gran}, Higher commutator conditions for extensions in Mal'tsev categories, \textit{Journal of Algebra} \textbf{515} (2018), 298--327.
	
	\bibitem{Everaert}
	\textsc{T. Everaert}, Higher central extensions in Mal'tsev categories, \textit{Applied Categorical Structures}  \textbf{22} (2014), 961--979.
	
	\bibitem{EGV}
	\textsc{T. Everaert, M. Gran and T. Van der Linden}, Higher Hopf formulae for homology and Galois theory, \textit{Advances in Mathematics} \textbf{217} (2008), 2231--2267.
	
	\bibitem{Findlay}
	\textsc{G.D. Findlay}, Reflexive homomorphic relations, \textit{Canadian Mathematical Bulletin} \textbf{3} (1960), 131--132.
	
	\bibitem{Gran}
	\textsc{M. Gran}, Internal categories in Mal'cev categories, \textit{Journal of Pure and Applied Algebra} \textbf{143} (1999), 221--229.
	
	\bibitem{Gran-Rodelo}
	\textsc{M. Gran and D. Rodelo}, A new characterization of Goursat categories, \textit{Applied Categorical Structures} \textbf{20} (2012), 229--238.
	
	\bibitem{Cuboid} \textsc{M. Gran and D. Rodelo}, The Cuboid Lemma and Mal'tsev categories, \textit{Applied Categorical Structures} \textbf{20} (2014), 805--816.
	
\bibitem{GR2}
	\textsc{M. Gran and D. Rodelo}, Beck-Chevalley conditions and Goursat categories, \textit{Journal of Pure and Applied Algebra } \textbf{221} (2017), 2445--2457.

\bibitem{HH} \textsc{J. Hagemann and A. Mitschke}, On n-permutable congruences, \textit{Algebra Universalis} \textbf{3} (1973), 8--12.

\bibitem{Huq} \textsc{S.A. Huq}, Commutator, nilpotency and solvability in categories, \textit{Quart. J. Math.} \textbf{19} (1968), 363--389.

\bibitem{jacqminthesis} \textsc{P.-A. Jacqmin}, Embedding theorems in non-abelian categorical algebra, \textit{Universit\'e catholique de Louvain thesis} (2016).

\bibitem{jacqmin1} \textsc{P.-A. Jacqmin}, An embedding theorem for regular Mal'tsev categories, \textit{Journal of Pure and Applied Algebra} \textbf{222 (5)} (2018), 1049--1068.

\bibitem{jacqmin2} \textsc{P.-A. Jacqmin}, Embedding theorems for Janelidze's matrix conditions, \textit{submitted to Journal of Pure and Applied Algebra}.

\bibitem{jacqmin3} \textsc{P.-A. Jacqmin}, Partial Mal'tsev algebras and an embedding theorem for (weakly) Mal'tsev categories, \textit{submitted to Cahiers de Topologie et G\'eom\'etrie Diff\'erentielle Cat\'egoriques}.

\bibitem{JR}
	\textsc{P.-A. Jacqmin and D. Rodelo}, Stability properties characterizing $n$-permutable categories, \textit{Theory and Applications of Categories} \textbf{32} (2017), 1563--1587.

\bibitem{Janelidze-cat} \textsc{G. Janelidze}, Internal categories in Mal'cev varieties, preprint, York University in Toronto (1990).

\bibitem{Double}  \textsc{G. Janelidze}, What is a double central extension? The question was asked by Ronald Brown, \textit{Cahiers de Topologie et G\'eom\'etrie Diff\'erentielle Cat\'egoriques} \textbf{32 (3)} (1991), 191--201.

\bibitem{JK}
	\textsc{G. Janelidze and G.M. Kelly}, Galois theory and a general notion of central extension, \textit{Journal of Pure and Applied Algebra} \textbf{2} (1994), 135--161.
	
	\bibitem{JanPed} \textsc{G. Janelidze and M.C. Pedicchio}, Pseudogroupoids and commutators,  \textit{Theory and Applications of Categories} \textbf{8} (2001), 408--456.
	
	\bibitem{JST} \textsc{G. Janelidze, M. Sobral and W. Tholen}, Beyond Barr exactness: Effective Descent Morphisms, in Categorical Foundations, edited by M.C. Pedicchio and W. Tholen, \textit{Encyclopedia of Mathematics and its Applications} \textbf{97}, Cambridge University Press (2004).	
	
\bibitem{janelidzeZ} \textsc{Z. Janelidze}, Subtractive categories, \textit{Applied Categorical Structures} \textbf{13} (2005), 343--350.
	
\bibitem{janelidzematrices} \textsc{Z. Janelidze}, Closedness properties of internal relations I: a unified approach to Mal'tsev, unital and subtractive categories, \textit{Theory and Applications of Categories} \textbf{16} (2006), 236--261.	
	
	\bibitem{Johnstone} \textsc{P.T. Johnstone}, Affine categories and naturally Mal'cev categories, \textit{Journal of Pure and Applied Algebra} \textbf{61} (1989), 251--256.
	
	\bibitem{JohnstoneH} \textsc{P.T. Johnstone}, The closed subgroup theorem for localic herds and pregroupoids, \textit{Journal of Pure and Applied Algebra} \textbf{70}, (1991) 97--106.
			
	\bibitem{Kock}
	\textsc{A. Kock}, 
	Generalized fibre bundles, \textit{Lecture Notes in Mathematics} \textbf{1348} (1988), 194--207.
	
	\bibitem{Lack}
	\textsc{S. Lack}, 
	The $3$-by-$3$ Lemma for regular Goursat categories, \textit{Homology Homotopy and Applications} \textbf{6} (2004), 1--3.
	
	\bibitem{Lambek58}
	\textsc{J. Lambek}, Goursat's theorem and the Zassenhaus lemma, \textit{Canadian Journal of Mathematics} \textbf{10} (1958), 45--56.
	
	\bibitem{Lambek92}
	\textsc{J. Lambek}, On the ubiquity of Mal'cev operations, \textit{Contemporary Mathematics} \textbf{131 (3)} (1992), 135--146.
	
	\bibitem{MLane} \textsc{S. Mac Lane}, Homology, Berlin-G\"ottingen-Heildelberg: Springer \textbf{114} (1963), 422pp.
	
	\bibitem{Maltsev}
	\textsc{A.I. Mal'tsev}, On the general theory of algebraic systems, \textit{Matematicheskii Sbornik, N.S.} \textbf{35 (77)} (1954), 3--20.
	
	\bibitem{Meisen}
	\textsc{J. Meisen}, Relations in categories, \textit{McGill University thesis} (1972).
	
	\bibitem{Mitschke}
	\textsc{A. Mitschke}, 
	Implication algebras are $3$-permutable and $3$-dis\-tri\-bu\-tive, \textit{Algebra Universalis} \textbf{1} (1971), 182--186.
	
	\bibitem{Ped1}
	\textsc{M.C. Pedicchio}, 
	A categorical approach to commutator theory, \textit{Journal of Algebra} \textbf{177} (1995), 647--657.

	\bibitem{Ped2}
	\textsc{M.C. Pedicchio}, 
	Arithmetical category and commutator theory, \textit{Applied Categorical Structures} \textbf{4} (1996), 297--305.

	 \bibitem{Riguet}
	\textsc{J. Riguet}, Relations binaires, fermetures, correspondances de Galois, \textit{Bulletin de la Soci\'et\'e Math\'ematique de France} \textbf{76} (1948), 114--155.
	
	\bibitem{Smith} \textsc{J.D.H. Smith}, Mal'cev varieties, \textit{Lecture Notes in Mathematics} \textbf{554} (1976).
	
	\bibitem{Werner}
	\textsc{H. Werner}, A Mal'cev condition for admissible relations, \textit{Algebra Universalis} \textbf{3} (1973), 263.
\end{thebibliography}
\end{document}